\documentclass[aihp]{imsart}
\usepackage{xcolor}
\RequirePackage{amsthm,amsmath,amsfonts,amssymb}
\RequirePackage[numbers]{natbib}
\usepackage{enumerate,url}
\usepackage{mathtools}

\startlocaldefs
\numberwithin{equation}{section}
\theoremstyle{plain}
\newtheorem{Thm}{Theorem}[section]
\newtheorem{Lem}[Thm]{Lemma}
\newtheorem{Prop}[Thm]{Proposition}
\newtheorem{Coro}[Thm]{Corollary}

\theoremstyle{remark}
\newtheorem{Rem}[Thm]{Remark}

\def\cal#1{\mathcal{#1}}
\newcommand{\comment}[1]{}
\def\fddto{\xrightarrow{\textit{f.d.d.}}}
\newcommand{\ind}{{\bf 1}}

\def\inddd#1{{\ind}_{\left\{#1\right\}}}

\newcommand{\proba}{\mathbb P}

\newcommand{\esp}{{\mathbb E}}

\newcommand{\inv}{^{-1}}

\newcommand{\eqnh}{\begin{eqnarray*}}
\newcommand{\eqne}{\end{eqnarray*}}
\newcommand{\eqnhn}{\begin{eqnarray}}
\newcommand{\eqnen}{\end{eqnarray}}
\newcommand{\equh}{\begin{equation}}
\newcommand{\eque}{\end{equation}}

\def\summ#1#2#3{\sum_{#1 = #2}^{#3}}

\def\sif#1#2{\sum_{#1=#2}^\infty}

\newcommand{\eqd}{\stackrel{d}{=}}

\def\topp#1{^{(#1)}}

\def\abs#1{\left|#1\right|}

\def\ccbb#1{\left\{#1\right\}}

\def\pp#1{\left(#1\right)}

\def\mmid{\;\middle\vert\;}

\def\floor#1{\left\lfloor #1 \right\rfloor}

\def\ceil#1{\left\lceil #1 \right\rceil}

\def\vv#1{{\boldsymbol #1}}
\def\vvc{{\boldsymbol c}}
\def\vvi{{\boldsymbol i}}

\def\vvx{{\vv x}}

\def\qmand{\quad\mbox{ and }\quad}

\def\qmwith{\quad\mbox{ with }\quad}
\def\mfa{\mbox{ for all }}

\def\mmas{\mbox{ as }}

\def\wt#1{\widetilde{#1}}
\def\wb#1{\overline{#1}}
\def\what#1{\widehat{#1}}

\def\limn{\lim_{n\to\infty}}

\def\limsupn{\limsup_{n\to\infty}}

\def\weakto{\Rightarrow}

\def\R{{\mathbb R}}

\def\N{{\mathbb N}}

\def\calB{\mathcal B}

\def\calD{\mathcal D}

\def\calF{\mathcal F}

\def\calG{\mathcal G}
\def\calH{\mathcal H}
\def\calI{\mathcal I}
\def\calJ{\mathcal J}

\def\calL{\mathcal L}
\def\calM{\mathcal M}

\def\calR{\mathcal R}

\def\cal#1{\mathcal{#1}}
\def\topp#1{^{\scriptscriptstyle (#1)}}

\def\sm{{\rm SM}}

\def\ddelta#1{\delta_{\pp{#1}}}
\def\PPP{{\rm PPP}}
\def\SM{{\rm SM}}

\endlocaldefs

\begin{document}

\begin{frontmatter}
\title{Phase transition for extremes of a family of stationary multiple-stable processes}

\runtitle{Extremes of multiple-stable processes}

\begin{aug}
\author[A]{\fnms{Shuyang} \snm{Bai}\ead[label=e1]{bsy9142@uga.edu}}
\and
\author[B]{\fnms{Yizao} \snm{Wang}\ead[label=e2]{yizao.wang@uc.edu}}

\address[A]{Department of Statistics \\ University of Georgia\\
310 Herty Drive\\
Athens, GA, 30602, USA., \printead{e1}}

\address[B]{Department of Mathematical Sciences\\University of Cincinnati\\2815 Commons Way\\Cincinnati, OH, 45221-0025, USA., \printead{e2}}
\end{aug}

\begin{abstract}
We investigate a family of stationary processes that may exhibit either long-range or short-range dependence, depending on the parameters. The processes can be represented as multiple stable integrals, and there are two parameters for the processes, the memory parameter $\beta\in(0,1)$ and the multiplicity parameter $p\in\N$. We investigate the macroscopic limit of extremes of the process, in terms of convergence of random sup-measures, for the full range of parameters. Our results show that (i) the extremes of the process exhibit long-range dependence when $\beta_p := p\beta-p+1\in(0,1)$, with a new family of random sup-measures arising in the limit, (ii) the extremes are of short-range dependence when $\beta_p<0$, with independently scattered random sup-measures arising in the limit, and (iii) there is a delicate phase transition at the critical regime $\beta_p = 0$.  
\end{abstract}

\begin{keyword}[class=MSC]
\kwd[Primary ]{60F17}
\kwd{60G70}
\kwd[; secondary ]{60G52}
\kwd{60K05}
\end{keyword}

\begin{keyword}
\kwd{stable regenerative set}
\kwd{random sup-measure}
\kwd{long-range dependence}
\kwd{infinite ergodic theory}
\kwd{phase transition}
\kwd{regular variation}
\kwd{multiple integral}
\kwd{renewal process}
\end{keyword}

\end{frontmatter}

\section{Introduction and main results}\label{sec:1}
\subsection{Background}
Extensive investigations on stationary stable processes started in 1980s. These processes extend naturally Gaussian processes, have regularly-varying tails, and they exhibit very rich dependence structures. The monograph by \citet{samorodnitsky94stable} summarizes foundational earlier developments of stable processes, and remains a classic on the topic. 

In early developments of stable processes, one of the main driving questions was to understand their ergodic properties. A modern approach in this direction was initiated by \citet{rosinski95structure}, who first related every stationary stable process to a dynamical system, more precisely an underlying flow on a certain space, and then revealed that the ergodic properties of the stationary stable process of interest can be derived from the one of the underlying flow. Then, instead of working on stable processes directly, one benefits from known results on dynamical systems, which shed light in turn on the studies of stable processes. 
Thanks to the seminal work of Rosi\'nski and another fundamental contribution later from \citet{samorodnitsky05null}, it is now well understood that all stationary stable processes can be  divided into three categories according to the ergodic-theoretic classification of the underlying flows, and each category of stable processes is with distinct ergodic properties. 
More precisely, the underlying flow can be positively recurrent, null recurrent 
 (a.k.a.~null-conservative) 
and dissipative: when the flow is dissipative, the corresponding stable process has a mixed moving average representation and is known to be mixing; when the flow is positively recurrent, the stable process is known to be non-ergodic; when the flow is null recurrent, the stable process is ergodic and weakly mixing, and yet it does not necessarily have the mixing property. 

From a different aspect, stable processes driven by 
positively- or null-recurrent flows are often associated with the notion of {\em long-range dependence} (a.k.a.~long memory), while those driven by dissipative flows are often associated with the notion of {\em short-range dependence}. The investigation of stochastic processes with long-range dependence has a long history \citep{beran13long,pipiras17long}, and in these studies stable processes provide important examples with representative properties. The recent monograph by \citet{samorodnitsky16stochastic} provided a first systematic presentation on stable processes (and more generally infinitely-divisible processes) incorporating the points of view of both ergodic-theoretic classifications and long-/short-range dependence.

Some later advances focus on further refinements of the aforementioned classifications of stable processes. When the underlying flow is dissipative, it is   known that the stable process is a mixed moving average process, and further classifications can be found in \citet{pipiras17stable}, summarizing mostly a series of extensive investigations by the same authors since 2000s. Within this class of stochastic processes with relatively simple structures, the situation is already very delicate. At the same time,  it was argued in \citet{samorodnitsky05null}   that the most challenging case is when the underlying flow is null-recurrent. In words, for processes within this category, while they are ergodic, they already exhibit abnormal asymptotic behaviors in terms of 
limit fluctuations.

Many recent advances on stable processes concern their limit theorems: for the partial-sum processes, for the extremes, and for total variations, just to mention a few. Sometimes, qualitatively different asymptotic behaviors are shown in complete correspondence to the aforementioned classifications based on ergodic properties, which is not surprising; but often much more elaborated information on the asymptotic behaviors in terms of limit theorems for fluctuations have been established, and new stochastic processes arise in such investigations. 
Our object of interest
 is a class
of stable processes
 driven by null-recurrent flows 
 and its extensions
  to be explained below. 
  We shall focus on limit theorems regarding their extremes, following the recent developments \citep{owada15maxima,lacaux16time,samorodnitsky19extremal}. For recent limit theorems of other types for stable and infinitely-divisible processes, we mention \citep{durieu20infinite,samorodnitsky16stochastic,owada16limit,basse17power,basse19limit,basse20berry}.

\subsection{Random sup-measures}
The asymptotic extremes at macroscopic level are 
commonly
 characterized by convergence of random sup-measures.  The foundation of random sup-measures and the corresponding weak convergence theory have been established in the 1980s and 1990s, and standard references are \citep{obrien90stationary,vervaat88random}.  We shall provide a brief introduction below where all the claims can be found in \citep{vervaat88random}. 
 
 In this paper, we shall focus on sup-measures  over the unit interval $[0,1]$ taking value in $\overline{\R}=[-\infty,\infty]$. In particular, suppose $\calG$ is the collection of open subsets of $[0,1]$ 
 (w.r.t.~subspace topology) .
  A sup-measure is a  mapping $m:\calG \mapsto \overline{\R}  $ that satisfies 
\[
m\pp{\bigcup_\lambda G_\lambda} = \sup_\lambda m(G_\lambda),
\]
for any family    $\{G_\lambda\}_\lambda$, $G_\lambda\in \calG$, with the convention $m(\emptyset) =   -\infty$. A sup-measure $m$ defined on open subsets admits a canonical extension to an arbitrary subset $A\subset [0,1]$ via $m(A):=\inf_{G\supset A, G\in \cal G} m(G)$.  The space of such sup-measures, denoted as $\sm=\sm([0,1],\overline{\R})$, can be equipped with the so-called sup vague topology, which is  generated by the subbase consisting of subsets of $\sm$  in the forms $\{m\in \sm: \ m(G)>x\}$ and $ \{m\in \sm: \ m(F)<x\}$ with $G\in \calG$,  $F\in \calF$ and $x\in \overline{\R}$,  where $\calF$ denotes the collection of closed 
 subsets of $[0,1]$. Equipped with the sup vague topology, the space  $\sm$  is   compact and second-countable,     and hence metrizable. 
 
A random sup-measure  $M$ is a random element taking value in the measurable space $(\sm,\calB(\sm))$, where $\calB(\sm)$ denotes the Borel $\sigma$-field on $\sm$ with respect to the sup vague topology. In addition, for each $A\in \calG\cup \calF$, the mapping $\sm \mapsto \wb{\R},  m \mapsto m(A)$ is measurable and hence each $M(A)$ is a ($\wb{\R}$-valued) random variable.   Suppose $\calG_0$ is the collection of all nonempty open subintervals of $[0,1]$. Then     a  mapping $M$ from the underlying probability space to  $\sm$ is a random sup-measure if and only if every $M(I)$ is a random variable for any $I\in\calG_0$.  Furthermore, the law of $M$ in $(\sm,\calB(\sm))$ is uniquely determined by the finite-dimensional distributions of the set-indexed process $\{M(I)\}_{I\in \calG_0}$. At last, the weak convergence of a sequence of random sup-measures 
 $\{M_n\}_{n\in\N}$ 
  to a limit random sup-measure $M$   in  $(\sm,\calB(\sm))$, 
 denoted by 
  \[
  M_n\weakto M,
  \]
 in $(\SM,\calB(\SM))$ as $n\to\infty$,
is equivalent to 
\begin{equation}\label{eq:fdd RSM conv}
\ccbb{M_n(I)}_{I\in\calG_0^*}\fddto\ccbb{M(I)}_{I\in\calG_0^*}
\end{equation}
as $n\rightarrow\infty$, where 
  the convergence above means convergence of finite-dimensional distributions    of set-indexed stochastic processes, and $\calG_0^*=\{I\in \calG_0:\  \proba\left(M(I)=M(\mathrm{cl}(I)\right)=1\}$, with $\mathrm{cl}(I)$ denoting the closure of $I$. In fact, for all the limit random sup-measures we shall encounter in this paper,  we have $\calG_0^*=\calG_0$.

As an  example of random sup-measure,  consider first
\equh\label{eq:isRSM}
\calM_{\alpha}^{\rm is}(G):=\sup_{i\in\N}\frac1{\Gamma_i^{1/\alpha}}\inddd{U_i\in G} , G\in\calG, 
\eque
where  $\alpha>0$, $\{\Gamma_i\}_{i\in\N}$ are consecutive arrival times of a standard Poisson process, and $\{U_i\}_{i\in\N}$ are i.i.d.~uniform random variables over $(0,1)$ independent from $\{\Gamma_i\}_{i\in\N}$.  
 The random sup-measure  $\calM_{\alpha}^{\rm is}$  is known as an
  {\em independently scattered $\alpha$-Fr\'echet random sup-measure} with Lebesgue control measure on $[0,1]$. Here, it is $\alpha$-Fr\'echet in the sense that $\proba(\calM_\alpha^{\rm is}(G) \le x) = e^{-{\rm Leb}(G)x^{-\alpha}}$ for all $x>0,G\in\calG$, and it is independently scattered in the sense that for disjoint $G_1,\ldots, G_m\in\calG$,  $\calM_\alpha^{\rm is}(G_1),\ldots,\calM_\alpha^{\rm is}(G_m)$ are independent. It is a classical result that, for i.i.d.~random variables $\{X_k\}_{k\in\N}$ with $\proba(X_k>x)\sim x^{-\alpha}$ as $x\rightarrow\infty$, 
\equh\label{eq:fdd RSM}
\frac1{c_n}M_n  \weakto \calM_\alpha^{\rm is}  \qmwith M_n(G) :=\max_{k/n\in G}X_k,\quad G\in\cal{G},
\eque
with $c_n = n^{1/\alpha}$;  here and below, we understand a maximum or a supremum over an empty set as $-\infty$.  

For a general stationary sequence of random variables, the above \eqref{eq:fdd RSM} may continue to hold   with possibly different normalizations $c_n$ and also different random sup-measures in place of $\calM_\alpha^{\rm is}$. 
Heuristically, the dependence of extremes of $\{X_k\}_{k\in\N}$ is considered weak, or of local nature, if the same $\calM_\alpha^{\rm is}$ arises in the limit with possibly a different normalization $c_n$. 
\citet{obrien90stationary} advocated using such a framework as \eqref{eq:fdd RSM} to characterize macroscopic limit of extremes of a stationary sequence when the limit is {\em not independently scattered}. 
Therein,  a thorough characterization of all possible limit random sup-measures (from a general sequence not necessarily with regularly varying tails) was carried out, and in particular, the 
shift-invariance
 and self-similarity properties of the limit are established.
For the model of our interest later, we shall see that independently scattered Fr\'echet random sup-measures as well as another family, denoted by $\calM_{\alpha,\beta,p}$ below, may arise in the limit, depending on the range of the parameters of the model.  

For readers familiar with functional central limit theorems (Donsker's theorem) for the partial-sum process but not as much with extremal limit theorems, it is worth noting that the convergence of random sup-measures reveals strictly more information than the convergence of the partial-maximum process
\[
\frac1{c_n}\ccbb{\max_{k=1,\dots,\floor{nt}} X_k}_{t\in[0,1]} \fddto \{\mathbb M_t\}_{t\in[0,1]},\quad n\rightarrow\infty.
\]
This is because different limits of random sup-measures may correspond to the same time-indexed process $\{\mathbb M_t\}_{t\in[0,1]}$ \citep[Proposition A.1]{durieu18family}.

\subsection{Stable-regenerative models}\label{sec:stable reg intro}
The model of our interest stems from a family of stable processes introduced first by \citet{rosinski96classes}. 
 The original model can be viewed as a prototype of stable processes driven by a null-recurrent flow. The investigations of its limit theorems dated back to the early 2000s \citep{resnick00growth,resnick04point,samorodnitsky04extreme}. Key advances appeared recently in \citep{owada15functional,owada15maxima,jung17functional,samorodnitsky19extremal}, and it is understood since these works that the large-scale behavior of this family is closely related to \emph{stable-regenerative sets} \citep{bertoin99subordinators}, and hence we name the original discrete-time stable processes {\em stable-regenerative stable processes} or {\em stable-regenerative models}. (We prefer     the latter as the two `stable' in the former name may cause   confusion.) 

 Let $\{X_k\}_{k\in\N}$ denote such a stable-regenerative model. This is a symmetric $\alpha$-stable process a parameter $\alpha\in(0,2)$ and a memory parameter $\beta\in(0,1)$. The representation is intrinsically related to renewal processes, and for which we introduce some notations. 
Consider a discrete-time renewal process starting at the origin with the consecutive renewal times denoted by $\vv\tau := \{\tau_0,\tau_1,\tau_2,\dots\}$, 
where $\tau_0=0<\tau_1<\tau_2<\ldots$. 
We let $F$ denote the renewal distribution
function,
 that is $F(x) = \proba(\tau_{i+1}-\tau_i\le x), i\in\N_0:=\{0,1,2,\ldots.\}, x\in\N$. Throughout, we assume 
\equh\label{eq:F}
\wb F(x) = 1-F(x) \sim \mathsf C_Fx^{-\beta} \mmas x\to\infty \qmwith \beta\in(0,1),
\eque
and the following technical assumption, with $f(n) :=
 F(n) - F(n-1)$ denoting the probability mass function of the renewal distribution,
\equh\label{eq:Doney}
\sup_{n\in\N}\frac{n f(n)}{\wb F(n)}<\infty.
\eque
In this paper a renewal process  is often {\em shifted} so that the starting point may not be the origin. 
An important notion is the {\em stationary shift distribution} 
 of the renewal process, denoted by $\pi$. Since the renewal distribution has infinite mean, the stationary shift 
 distribution
  $\pi$ is a $\sigma$-finite and infinite measure on $\N_0$ unique up to a multiplicative constant. 
We shall work with 
\equh\label{eq:pi}
\pi(\{k\}) 
:=
 \wb F(k), \quad k\in\N_0.
\eque
Then it is well-known that the law of the shifted renewal process $d+\vv\tau:=\{d,d+\tau_1,d+\tau_2,\ldots\}$   is shift-invariant (see, e.g., \cite{bai21tail}). Here by the  law  of  
\[
\vv{\tau}^*:=d+\vv\tau,
\]
we mean the  pushforward    of   the  product measure between $\pi$ and the law of $\vv \tau$. 
We shall also need
\equh\label{eq:w_n}
w_n 
:=
 \summ k1n \pi(\{k\})  = \sum_{k=1}^n\wb F(k)\sim  \frac{\mathsf C_F}{1-\beta} n^{1-\beta}  \text{ as }n\rightarrow\infty.
\eque

We are now ready 
to
 define the stable-regenerative model through a series representation. Let $\pi$ be as in \eqref{eq:pi}, and consider
\[
\sif i1\ddelta{\eta_i,d_i}\eqd \PPP\pp{
 (0,\infty] 
\times\N, \alpha x^{-\alpha-1}dxd\pi}.
\]
Suppose 
$\{\vv \tau\topp i\}_{i\in \N}$ are i.i.d.\ copies of the non-shifted renewal process $ \vv \tau$   which are independent of the point process above.
Set $\vv \tau\topp{i,d_i} 
:=
d_i+\vv \tau\topp i$, $i\in \N$. Then, the stable-regenerative model is defined as
\equh\label{eq:series infty}
\ccbb{X_k}_{k\in\N}=\ccbb{\sif i1 \varepsilon_i\eta_i\inddd{k\in \vv\tau\topp{i,d_i}}}_{k\in\N},
\eque
where $\{\varepsilon_i\}_{i\in\N}$ are i.i.d.~Rademacher random variables independent from all other random elements previously introduced.

The model \eqref{eq:series infty} exhibits long-range dependence in terms of limit theorems as revealed in \citep{owada15functional,owada15maxima,samorodnitsky19extremal}. In particular, it was shown in \citep{samorodnitsky19extremal} that  
 random sup-measures with long-range clustering, denoted by $\calM_{\alpha,\beta}$,
arise
 as macroscopic limits of extremes of \eqref{eq:series infty}. Thanks to translation invariance and self-similarity, it suffices to restrict to random sup-measures on $[0,1]$, and we do so throughout.

The random sup-measure $\calM_{\alpha,\beta}$ is built upon independent stable regenerative sets that we now recall some basics. 
A $\beta$-stable-regenerative set, say $\wt\calR_\beta$, can be defined as
the scaling limit of $\vv\tau/n$, viewed as a random closed set, as $n\to\infty$, with $\vv\tau$ as the set of renewal times of non-shifted renewal processes with renewal distribution corresponding to \eqref{eq:F}; or, $\wt\calR_\beta$ can be defined as 
 the closure of the image of a $\beta$-stable subordinator which satisfies   $0\in \wt\calR_\beta$ almost surely. 
Then, by a {\em randomly shifted $\beta$-stable regenerative set} we refer to
\[
\calR_\beta:=\wt\calR_\beta+B_{1-\beta,1},
\] where $B_{1-\beta,1}$ is a Beta($1-\beta,1)$ distributed shift ($\proba(B_{1-\beta,1}\le x) = x^{1-\beta}, x\in[0,1]$) independent from $\wt\calR_\beta$. The stable-regenerative sets have very rich structures and many fundamental developments took place in 1980s and 1990s \citep{fitzsimmons85intersections,fitzsimmons88stationary,fitzsimmons85set,bertoin99subordinators}. 
They are important examples of random closed sets, of which our reference is \citep{molchanov17theory}.
The appearance of $B_{1-\beta,1}$ is necessary  to make the randomly shifted stable-regenerative sets stationary 
(a.k.a.~shift-invariant)
in an appropriate sense, and this has been well understood (see e.g.\ \citep{samorodnitsky19extremal}). 

Let
$\{\calR_{\beta,i}\}_{i\in\N}$
denote i.i.d.~copies of $\calR_\beta$. 
We also need the following (see \citep{samorodnitsky19extremal}) regarding their intersections:
\equh\label{eq:p'}
\bigcap_{i=1}^q\calR_{\beta,i} \eqd
\begin{cases}
 \calR_{\beta_q}, & \mbox{ if } q=1,\dots,p',\\
 \emptyset, & \mbox{ otherwise,}
 \end{cases}
 \qmwith p':=\max\{q\in\N:\beta_q> 0\}.
\eque
Now we introduce $\calM_{\alpha,\beta}$. Note that our representation below is not the one used in \citep{samorodnitsky19extremal} as the definition, but it was already used in the proofs therein. Recall $\{\varepsilon_i\}_{i\in\N}$ are i.i.d.~Rademacher random variables.  Let $\{\Gamma_i\}_{i\in\N}$ be consecutive  arrival times of   a standard Poisson process on $[0,\infty)$. Suppose the three sequences $\{\calR_{\beta,i}\}_{i\in\N}$, $\{\varepsilon_i\}_{i\in\N}$ and $\{\Gamma_i\}_{i\in\N}$ are independent of each other. We set
\equh\label{eq:RSM cluster}
\calM_{\alpha,\beta}(G) := \sup_{J\subset\N, |J|\le p'}\calM_{\alpha,\beta,J}(G)
\qmwith 
 \calM_{\alpha,\beta,J}(G) := 
\begin{cases}
\displaystyle\sum_{i\in J}\frac{\varepsilon_i}{\Gamma_i^{1/\alpha}}, & \mbox{ if } \calR_{\beta,J}\cap G\ne\emptyset,\\
-\infty, & \mbox{ otherwise.}
\end{cases}
\eque
We interpret that each $\calM_{\alpha,\beta,J}$ representing an {\em aggregated cluster} of  {\em individual clusters}, with each  individual  cluster indexed by $i\in J$ having magnitude $\varepsilon_i\Gamma_i^{-1/\alpha}$ and locations $\calR_{\beta,i}$. The aggregated cluster then has magnitude and locations represented by
\equh\label{eq:aggregated cluster p=1}
\sum_{i\in J}\frac{\varepsilon_i}{\Gamma_i^{1/\alpha}} \qmand \calR_{\beta,J} = \bigcap_{i\in J}\calR_{\beta,i}. 
\eque
We interpret the clustering as {\em long-range clustering}, as the locations are represented by unbounded random sets.
\begin{Rem}
  Note that for our convention, the aggregated clusters and individual clusters may have negative values, and hence they do not represent extremal clusters in the usual sense. But our interpretation is convenient when extending the results later.
\end{Rem}

When $\beta\le 1/2$, $\calM_{\alpha,\beta}$ is an $\alpha$-Fr\'echet random sup-measure, in the sense that  for all $G_1,\dots,G_d\in\calG$, $(\calM_{\alpha,\beta}(G_i))_{i=1,\dots,d}$ has multivariate $\alpha$-Fr\'echet distribution, that is, for all $a_1,\dots,a_d>0, \max_{i=1,\dots,d}a_i\calM_{\alpha,\beta}(G_i)$ has an $\alpha$-Fr\'echet distribution \citep{stoev06extremal},
 and yet it is no longer  independently scattered due to the presence of long-range clustering. Note that $p'=1$ when $\beta\le1/2$ so there is no aggregation in this regime. However, with $\beta>1/2$, due to the appearance of aggregation even the   marginal distribution $\calM_{\alpha,\beta}(G)$ is not $\alpha$-Fr\'echet. 

The random sup-measure $\calM_{\alpha,\beta}$ first appeared in \citep{lacaux16time} with $\beta\in(0,1/2)$ and \citep{samorodnitsky19extremal} with $\beta\in(0,1)$. It was proved therein that, with the empirical random sup-measure $M_n$ as in \eqref{eq:fdd RSM},  as $n\rightarrow\infty$,
\[
\frac1{w_n^{1/\alpha}} M_n    \weakto \calM_{\alpha,\beta}.
\]

\subsection{A multiple-stable version of stable-regenerative models}

The model of   interest in this paper is the following   extension of \eqref{eq:series infty}.  With the same ingredients defining \eqref{eq:series infty} above, consider
\equh\label{eq:series infty p}
\ccbb{X_k}_{k\in\N} =\ccbb{\sum_{0<i_1<\cdots<i_p} [\varepsilon_\vvi][\eta_\vvi]\inddd{k\in \bigcap_{r=1}^p\vv\tau\topp{i_r,d_{i_r}}}}_{k\in\N},
\eque
where $[\varepsilon_\vvi] = \varepsilon_{i_1}\times\cdots\times \varepsilon_{i_p}$, and similar notation for $[\eta_\vvi]$. 
It is known that such multiple series 
converges  almost surely and unconditionally 
(e.g.,
\citep[Theorem 1.3 and Remark 1.5]{samorodnitsky89asymptotic}).  Note that when $p=1$ it recovers the model \eqref{eq:series infty}. 

From now on, we   shall use the term \emph{stable-regenerative model} to refer to the more general class of models \eqref{eq:series infty p}.
The stable-regenerative model and its variants  have been investigated  in literature recently. In particular, when $p\ge 2$  and
 $\beta_p=p\beta -p+1\in(0,1)$,
  a functional central limit theorem has been established for \eqref{eq:series infty p} in \citep{bai20functional}, where the limits are a new family of a self-similar multiple-stable process with stationary increments.  See also \citep{bai20limit,bai22limit} for   variations of the model that scale  to multiple-Gaussian processes known as Hermite processes. 
 Our main motivation is to investigate the case $\beta_p\le 0$, about which little has been known in the literature. 
\begin{Rem}
It might be helpful to interpret the stable-regenerative model as follows. Consider the case $p=1$ in \eqref{eq:series infty} for the sake of simplicity. Then, one may view each renewal process $\vv\tau\topp{i,d_i}$ recording the visit times of a certain state of a null-recurrent Markov chain, and each value $\varepsilon_i\eta_i$ is the reward collected by the chain when the Markov chain visits the state. This interpretation has a flavor of aggregated random walks in random scenery, as explained in \citep{wang22choquet}. 
\end{Rem}

 \begin{Rem}
It is well-known that a series representation as in \eqref{eq:series infty p} can be alternatively expressed as a multiple stable integral, whence the name {\em multiple-stable processes}. We sketch the representation, starting with $p=1$. In this case, \eqref{eq:series infty} has the equivalent stochastic-integral representation
\[
\ccbb{X_k}_{k\in\N} \eqd C_\alpha\ccbb{\int _{\wt\Omega\times\N}\inddd{k\in s+\wt{\vv\tau}(\wt\omega)}M_\alpha(d\wt\omega ds)}_{k\in\N},
\]
for some explicit constant $C_\alpha$,
where   the $(\wt\Omega,\wt\calF,\wt\proba)$ is a  measure space equipped with a probability measure $\wt\proba$ (different from the underlying probability space), and moreover on $(\wt\Omega,\wt\calF,\wt\proba)$, $\wt{\vv\tau} = \{\wt \tau_0,\wt\tau_1,\wt\tau_2,\dots\}$, $\wt \tau_0=0<\wt \tau_1<\wt \tau_2<\ldots$, is a  renewal process with renewal distribution $F$ as before,  and   $M_\alpha$ is a symmetric $\alpha$-stable random measure on $\wt\Omega\times\N$ with control measure $\wt\proba\times \pi$, with $\pi$ as in \eqref{eq:pi}.  The above equivalence holds only for $\alpha\in(0,2)$ \citep{samorodnitsky94stable}, while the stochastic integral representation is valid for $\alpha=2$, and in this case the stochastic integral is with respect to a Gaussian random measure. 

The multiple-stable integral representations  would take more effort to explain. In particular,   the integral {\em excludes} the diagonals.
Since the development of our proofs 
relies
 exclusively on the series representation, we choose not to further elaborate the multiple-stable-integral representation, but refer the interested readers to  \citep[Section 3.1 and Example 4.2]{bai20functional} for more details.
\end{Rem}
 
\begin{Rem}For processes represented by multiple-integrals with respect to Gaussian random measures, most commonly known as Gaussian chaos, the literature is extensive. However, much less   has been devoted to multiple-stable (non-Gaussian) processes since early investigations in the 1980s and 1990s  \citep{samorodnitsky89asymptotic,rosinski99product,krakowiak86random,kwapien92random}. In most of the aforementioned papers, the focus is rather on the properties of multiple-stable integrals, instead of the stochastic processes represented by such integrals. 

We also mention that the multiple-stable process \eqref{eq:series infty p} is of its own interest as it has led to a new representation of Hermite processes in terms of stable regenerative sets, 
their joint intersections, and the corresponding joint local times
 \citep{bai20representations}. 

\end{Rem}

\subsection{Main results: phase transition for extremes}
Our main results reveal a phase transition on the asymptotic behaviors of extremes for the multiple-stable version of stable-regenerative models. Before stating the limit theorems, we first introduce the limit random sup-measures at the super-critical regime, denoted by $\calM_{\alpha,\beta,p}$ below. This extends the limit random sup-measures in \citep{samorodnitsky19extremal} corresponding to $p=1$ here.

For $p\ge 2$, we extend the definition  \eqref{eq:RSM cluster} as follows.  
Throughout, denote
\[
\calD_p=\ccbb{\vvi=(i_1,\ldots,i_p)\in \ \N^p:\ i_1<\cdots<i_p}.
\]
Let   $\{\calR_{\beta,i}\}_{i\in\N}$, $\{\varepsilon_i\}_{i\in\N}$ and $\{\Gamma_i\}_{i\in\N}$ be as introduced before \eqref{eq:RSM cluster}.
The point process 
\[
\sum_{\vvi\in\calD_p}\ddelta{[\varepsilon_\vvi]/[\Gamma_\vvi]^{1/\alpha},\calR_{\beta,\vvi}}
\]
encodes the magnitude and locations of each {\em individual cluster} now indexed by $\vvi \in \calD_p$.
 To describe the index set of each {\em aggregated cluster} is a little more involved than in the case $p=1$. This time, each aggregated cluster is indexed by some $\vv c = (c_1,\dots,c_q)\in\calD_q$ for some $q\ge p$ as follows. Introduce the multi-index set  
\[
J(\vvc):=\ccbb{\vvi=(i_1,\ldots,i_p)\in\calD_p:\vvi\subset\vvc},
\]
where   by $\vvi\subset\vvc$ we mean $\{i_1,\dots,i_p\}\subset \{c_1,\dots,c_q\}$.
In particular, $|J(\vvc)| = \binom qp$.  For example, if $p=3, q=4$, then $J((1,3,5,6)) = \{(1,3,5),(1,3,6),(1,5,6),(3,5,6)\}$. 
Next, define the following countable collection of index sets 
\begin{equation}\label{eq:Jpp}
\calJ_{p,p'} := \ccbb{J(\vvc):\vvc\in\calD_q, q=p,\dots,p'} ,
\end{equation}
where $p'$ is as in \eqref{eq:p'}.  
Now  we say  each $J\in \calJ_{p,p'}$ 
 corresponds to an aggregated cluster, of which the magnitude and locations are represented by
\equh\label{eq:aggregated cluster}
\sum_{\vvi\in J}\frac{[\varepsilon_\vvi]}{[\Gamma_\vvi]^{1/\alpha}} \qmand \calR_{\beta,J} = \bigcap_{\vvi\in J}\calR_{\beta,\vvi},
\eque
respectively.
This time, the contributions to the aggregated cluster are from  individual clusters  indexed  by $\vvi\in J$. 
Note that by our choice of $p'$, $\calR_{\beta,J}\ne\emptyset$ almost surely for all $J\in \calJ_{p,p'}$  (see \eqref{eq:p'}). The expressions in \eqref{eq:aggregated cluster} when $p=1$ can be identified with those in \eqref{eq:aggregated cluster p=1}   (in particular $\calJ_{1,p'} = \{J\subset \N:\ 1\le |J|\le p'\}$).

We now introduce
\equh\label{eq:RSM}
 \calM_{\alpha,\beta,p}(G)  := \sup_{J\in\calJ_{p,p'}}\calM_{\alpha,\beta,J}(G)\qmwith \calM_{\alpha,\beta,J}(G) :=  \begin{cases}
\displaystyle 
\sum_{\vvi\in J}\frac{[\varepsilon_\vvi]}{[\Gamma_\vvi]^{1/\alpha}}, & \mbox{ if } \calR_{\beta,J}\cap G\ne\emptyset,\\
-\infty, & \mbox{ otherwise,}
\end{cases} 
\quad  G\in\calG.
\eque
Note that we use the notation $\calM_{\alpha,\beta,J}$ both here and previously in \eqref{eq:RSM cluster}. The two places are consistent in the sense that from now on we view $\calM_{\alpha,\beta,J}$ in \eqref{eq:RSM cluster} as a special case of the above with $p=1$.
\begin{Rem}
The definition  of $\calM_{\alpha,\beta,p}$  in \eqref{eq:RSM}   readily yields  a random sup-measure. 
 Indeed,  it is easily verified that each  $\calM_{\alpha,\beta,J}$ in \eqref{eq:RSM} is  a random sup-measure which takes 2 values, and hence so is the countable sup of $\calM_{\alpha,\beta,J}$ in \eqref{eq:RSM}.   It follows from that fact  $\proba(x\in \calR_{\beta,i})=0$ for any $x\in [0,1]$ (e.g., \cite[Proposition 1.9]{bertoin99subordinators}) that the   property $\calG^*_0=\calG_0$ mentioned below \eqref{eq:fdd RSM conv}  holds for $\calM_{\alpha,\beta,p}$.   An alternative representation of $\calM_{\alpha,\beta,p}(G)$ is provided in Section \ref{sec:agg clusters} which may shed light on the arise  of $\calM_{\alpha,\beta,p}$ as the limit in the super-critical regime.
\end{Rem}
\begin{Rem}\label{rem:phase transition super-crit}
In view of  \eqref{eq:RSM},
the number $p'$ defined in \eqref{eq:p'}  may be called the largest `size' of an aggregated cluster given $\beta$ and $p$. From \eqref{eq:p'}, we see that there is another layer of phase transitions within the supercritical regime $\beta>1-1/p$, in the sense that as $\beta$ increases towards $1$,   the  maximum `size' $q'$ of the aggregated cluster will also increase to infinity.
\end{Rem}
\begin{Rem}\label{rem:thinning}
In the case $p=1$, we have a further simplified representation as follows.  
First, as an alternative representation of \eqref{eq:RSM cluster} is  
\[
\calM_{\alpha,\beta,1}(G) = \sup_{J\subset\N,|J|\le p'}\calM_{\alpha,\beta,J}^+(G)\qmwith 
\calM_{\alpha,\beta,J}^+(G) := \begin{cases}
\displaystyle \sum_{i\in J}\frac{(\varepsilon_i)_+}{\Gamma_i^{1/\alpha}}, &\mbox{ if } \calR_{\beta,J}\cap G\ne\emptyset,\\
-\infty, & \mbox{ otherwise.}
\end{cases}
\]
once one realizes that for those $i\in\N$ such that $\varepsilon_i = -1$, $\varepsilon_i/\Gamma_i^{1/\alpha}$ has no contribution to the limit. Then, one may further drop all these terms by a standard thinning argument of Poisson point process, and arrive at
\begin{align*}
\ccbb{\calM_{\alpha,\beta,1}(G)}_{G\in\calG} &\eqd \ccbb{2^{-1/\alpha}\sup_{J\subset\N, |J|\le p'} \sum_{i\in J}\frac1{\Gamma_i^{1/\alpha}}\inddd{\calR_{\beta,J}\cap G\ne\emptyset}}_{G\in\calG}\\
 &\eqd \ccbb{2^{-1/\alpha}\sup_{t\in G}\sum_{i\in\N}\frac 1{\Gamma_i^{1/\alpha}}\inddd{t\in \calR_{\beta,i}\cap G\ne\emptyset}}_{G\in\calG}.
\end{align*}
For $p \ge 2$, such a simplified representation without any Rademacher random variables seems no longer available. Instead, since $\calM_{\alpha,\beta,p}(G)>0$ for all non-empty $G\in\calG$ (see Proposition \ref{pro:RSM}), we have the following simplification: \[
\calM_{\alpha,\beta,p}(G) = \sup_{J\in\calJ^+_{p,p'}}\calM_{\alpha,\beta,J}(G) \qmwith \calJ^+_{p,p'} := \ccbb{J\in\calJ_{p,p'}: \sum_{\vvi\in J}\frac{[\varepsilon_\vvi]}{[\Gamma_\vvi]^{1/\alpha}}>0}.
\]
In words, those individual clusters with negative magnitudes do not have any impact on the law of $\calM_{\alpha,\beta,p}$, and those indexed by $J\in\calJ_{p,p'}^+$ are the (long-range) {\em extremal clusters} in the common sense. 
\end{Rem}

The main result of this paper is the following. 
\begin{Thm}\label{thm:RSM}
We have the follow weak convergence in the    space of sup-measures  
 $\sm([0,1], \overline{\R})$: as $n\rightarrow\infty$,
\[
\frac1{c_n} M_n  \weakto \begin{cases}
 \mathfrak C_{F,p}^{1/\alpha}\calM_{\alpha,\beta,p} , & \mbox{ if } \beta_p>0,\\\\
 \mathfrak C_{F,p}^{1/\alpha}\calM_{\alpha}^{\rm is}   , & \mbox{ if } \beta_p\le  0,
\end{cases}
\]
with 
\[
c_n = \begin{cases}
\displaystyle n^{(1-\beta_p)/\alpha}, & \mbox{ if } \beta_p>0,\\\\
\displaystyle \pp{\frac{n(\log\log n)^{p-1}}{\log n}}^{1/\alpha}, & \mbox{ if } \beta_p=0,\\\\
(n\log ^{p-1}n)^{1/\alpha}, & \mbox{ if } \beta_p<0,
\end{cases}
\]
the random sup-measures  $\calM_{\alpha,\beta,p}$ in \eqref{eq:RSM} and $\calM_{\alpha}^{\rm is}$ in  \eqref{eq:isRSM},
and the constant 
$\mathfrak C_{F,p}$,
depending on $F$ (and hence $\beta$) and $p$ only,
to be specified in \eqref{eq:C super-crit}, \eqref{eq:thm sub-crit} and \eqref{eq:thm crit} in later sections  (where we shall restate the limit theorem before the proof in each regime again for convenience).
\end{Thm}

As a corollary, we summarize the marginal limit theorem for the   partial maxima. 
  Note that $\calM_{\alpha,\beta,p}([0,1])$
 is almost surely equal to the random variable \equh\label{eq:Z}
\mathsf Z_{\alpha,\beta,p}:= \sup_{J\in\calJ_{p,p'}}\sum_{\vvi\in J}\frac{[\varepsilon_\vvi]}{[\Gamma_\vvi]^{1/\alpha}},
\eque
because each $\calR_{\beta,J}\cap [0,1]\ne\emptyset$ for all $J\in \calJ_{p,p'}$ (see \cite[Corollary B.3]{samorodnitsky19extremal}). 
Some further simplification may be obtained. If $p'=p$, then $\mathsf Z_{\alpha,\beta,p} = \sup_{\vvi\in\calD_p}[\varepsilon_\vvi]/[\Gamma_\vvi]^{1/\alpha}$, and if $p'=p=1$ (which is exactly the case $p=1,\beta\in(0,1/2]$), then $\mathsf Z_{\alpha,\beta,1}\eqd 2^{-1/\alpha}\Gamma^{-1/\alpha}$ (see Remark \ref{rem:thinning}). 

\begin{Coro}
With $c_n$ and $\mathfrak C_{F,p}$ as above,
we have
\[
\limn\proba\pp{\frac1{c_n}\max_{k=1,\dots,n}X_k \le x} = \begin{cases}
\proba(\mathfrak C_{F,p}^{1/\alpha}\mathsf Z_{\alpha,\beta,p}\le x), & \mbox{ if } \beta_p>0,\\
\exp\pp{-\mathfrak C_{F,p} 
x^{-\alpha}}, & \mbox{ if } \beta_p\le 0,
\end{cases}
\]
with $\mathsf Z_{\alpha,\beta,p}$ defined in \eqref{eq:Z}. In particular, $\mathsf Z_{\alpha,\beta,p}$ is not an $\alpha$-Fr\'echet random variable unless $p=1,\beta\in(0,1/2]$. 
\end{Coro}
There are three different regimes as illustrated by Theorem \ref{thm:RSM}.
\begin{enumerate}[(1)]
\item  Super-critical 
 regime: $\beta_p>0$. The limit theorem of this regime is a generalization of \citep{samorodnitsky19extremal}, which corresponds to $p=1,\beta_p = \beta$ here, and in this case $\calM_{\alpha,\beta} = \calM_{\alpha,\beta,1}$.  It is worth emphasizing that not only $\calM_{\alpha,\beta,p}$ is not independently scattered with $\beta_p>0$, but also that except the case $p=1,\beta\in(0,1/2]$, $\calM_{\alpha,\beta,p}$ is not even $\alpha$-Fr\'echet, indicating very strong dependence of the original model. 
 The dependence is of an aggregation nature, 
 and a more refined phase transition within this regime regarding the size of the aggregated cluster of extremes
 is explained in Remark \ref{rem:phase transition super-crit}. 
\item Sub-critical 
 regime: $\beta_p<0$. 
 In this regime, the limit random sup-measure is independently scattered, and hence illustrates that the extremes of the original model are of short-range dependence. Note that it is a macroscopic property for the limit random sup-measure to be independently scattered, and in this case we obtain an immediate consequence that the well-known extremal index (e.g.~\citep{kulik20heavy}) of the process is $\theta = 2p!(p-1)! \mathfrak C_{F,p}=  \mathfrak q_{F,p} \mathsf D_{\beta,p}   \in(0,1)$, where $ \mathfrak q_{F,p}$ is given in \eqref{eq:c_Fp} and $\mathsf D_{\beta,p}$ is given in \eqref{eq:shape parameter}. 
 However, limit theorems for random sup-measures are unable to characterize precisely the 
 {\em local-clustering} of extremes. 
 As a more precise elaboration we investigate the tail processes \citep{basrak09regularly,kulik20heavy} in the accompanying paper \citep{bai21tail}.
\item Critical regime: $\beta_p = 0$. In this regime, the limit random sup-measure is again independently scattered, and yet in the limit theorem the normalization sequence $c_n$  is of strictly smaller order than $(n\log^{p-1}n)^{1/\alpha}$ as in the sub-critical regime.  It follows immediately that  the extremal index is zero, suggesting that the local clustering of extremes is unbounded (with infinite size). This again can be made precise by characterizing the tail processes as addressed in \citep{bai21tail}. 

The critical regime is of particular interest
 in the sense that it provides a rare example where at the microscopic level the extremes of the process do not exhibit typical behavior of short-range dependence (which normally has local clustering with finite size), and yet at the macroscopic level they do. 
 See Remark \ref{rem:TP} for more details. 
\end{enumerate}

\begin{Rem}\label{rem:TP}
Combined with the results in the accompanying paper \citep{bai21tail} on local clustering of extremes, we have a complete picture regarding limit extremes at both macroscopic and microscopic levels,  as summarized in Table \ref{table:1} below.
We first describe the tail processes and the limit theorem in the accompanying paper \citep{bai21tail}.
Let $\vv\Theta^*\equiv \{\Theta_k^*\}_{k\in\N_0}$ be a $\{0,1\}$-valued sequence defined as follows: let $\{\vv\tau\topp i\}_{i=1,\dots,p}$ denote i.i.d.~copies of a standard (non-shifted) renewal process, with  the renewal distribution function
 $F$ as in \eqref{eq:F}, and consider
\[
\Theta_{k}^* :=\begin{cases}
1, & \mbox{ if }  k\in\vv\eta,\\
0, & \mbox{ otherwise},
\end{cases}
\quad k=0,1,\dots \qmwith\vv\eta :=\bigcap_{r=1}^p \vv\tau\topp r.
\]
In particular, $\Theta_0^* = 1$ since $0\in\vv\tau\topp r, r=1,\dots,p$ by definition. 
Note that $\vv\eta$ is a non-shifted renewal process ($0\in\vv\eta$), and it is possibly terminating 
 (i.e., $\eta_1 = \infty$ with strictly positive probability; this is the case when the renewal distribution has a mass at infinity).
  The renewal process $\vv\eta$ is terminating if and only if $\beta_p<0$, and in this case with probability one, $\Theta_k^* = 0$ for all $k$ large enough. (More precisely, the terminating rate is shown to be the so-called candidate extremal index of the tail process; see \citep{bai21tail} for details.)
It is proved in \citep{bai21tail} that
for all $m\in\N$ (extending the definition of the stationary process to $\{X_n\}_{n\in\N_0}$ to follow the convention),
\[
\calL\pp{{\frac{X_0}{|X_0|},\dots,\frac{X_m}{|X_0|}}\mmid |X_0|>x} \to
\calL\pp{\varepsilon\Theta_0^*,\dots,\varepsilon\Theta_m^*},
\]
as $x\to\infty$, where $\varepsilon$ is a Rademacher random variable independent of $\{\Theta_{k}^*\}$. The left-hand side above is understood as the conditional law of the finite-dimensional distribution given $|X_0|>x$, and the right hand-side is the law of finite-dimensional distribution  the {\em spectral tail process}, $\{\varepsilon\Theta_k^*\}_{k\in\N_0}$,  of the model $\{X_n\}_{n\in\N_0}$.  We say the (spectral) tail process is terminating, if $\Theta_k^* = 0$ for $k$ large enough with probability one. Now we have the following summary of the phase transition at both levels.

\begin{table}[ht!]
\begin{tabular}{|c|c|c|c|c|c}
\hline
regime 
& tail process & limit random sup-measure\\
& (microscopic) & (macroscopic) \\
\hline
  super-critical, $\beta_p>0$ & non-terminating & $\calM_{\alpha,\beta,p}$ \\
   critical, $\beta_p=0$ &  non-terminating & $\calM^{\rm is}_\alpha$\\
    sub-critical, $\beta_p < 0$ & terminating & $\calM^{\rm is}_\alpha$\\
    \hline
\end{tabular}
\vspace{.2cm}
\caption{Summary of phase transition.}\label{table:1}
\end{table}

\end{Rem}
\begin{Rem}We are unaware of any other examples of phase transition of extremes for a stationary regularly-varying stochastic processes, except the recent result in \citep{durieu22phase} for the so-called heavy-tailed Karlin model with multiplicative noise.  It is worth noticing that therein the random sup-measures arising at the critical regime are different from the ones in both super- and sub-critical regimes therein, and also that the random sup-measures in both super-critical and critical regimes therein are different from $\calM_{\alpha,\beta,p}$ here. 
\end{Rem}

\subsection{Comments}
We conclude the introduction with several remarks regarding our main results, proofs, and follow-up questions. 

\begin{Rem}
We assume $\wb F$ to be regularly varying {\em without a slowly varying function that goes to either zero or infinity}. This constraint is for convenience only, as the proof is quite involved already. We do not expect the slowly varying function will change qualitatively the limiting objects in either super or sub-critical regime. In the critical regime, however, the limit tail process studied in the accompanying paper \cite{bai21tail} may have qualitatively different behaviors   when the limit of slowly varying function is not a finite non-zero constant. 
We shall leave the investigation of the effects of slowly-varying functions at the critical regime in a future work.
\end{Rem}
\begin{Rem}
One could also choose to work with multiple-infinitely-divisible process with other types of tails. If the tail decays like $C |x|^{-\alpha}$ at $|x|\to\infty$ for  $\alpha\in (0,\infty)$ (including our case here with $\alpha\in(0,2)$), the same phase transition shall occur (note that the limits $\calM_{\alpha,\beta,p}$ and $\calM_\alpha^{\rm is}$ are both valid for all $\alpha>0$). If the tail decays like $L(x) x^{-\alpha}$ for some slowly varying function, then this shall introduce some delicacy, but we do not expect an essential difference from our results here. We hence do not pursue such a generalization. A more challenging  problem is to consider other types of heavy tails that are not regularly varying. For $p=1$, 
\citet{chen22extremal,chen21new} recently revealed a very delicate phase transition from this respect: restricted to stable-regenerative model with heavy-tailed distributions (but not necessarily with regularly-varying tails), they proved that with different heaviness of the tails, different random sup-measures may arise in the limit.
\end{Rem}
\begin{Rem}
We provide here some  heuristic explanations  for each case of $\beta_p>0$ and $\beta_p\le 0$  to illustrate the different mechanisms of the formulation of extremes in the limit.

In the super-critical regime $\beta_p>0$,  working with \eqref{eq:series infty}, the key observation is  that      for every $i_1<\cdots<i_p$, the intersection $\bigcap_{r=1}^p\vv\tau\topp{i_r,d_{i_r}}$ scaled properly has a limit law as that of the shifted $\beta_p$-stable-regenerative set $\calR_{\beta_p}$. {\em In words, every collection of $p$ renewals indexed by $i_1<\cdots<i_p$ has a contribution to the limit random sup-measure.} 

For the sub-critical and critical regime $\beta_p\le 0$, $\bigcap_{r=1}^p\vv\tau\topp{i_r,d_{i_r}}$ is eventually an empty set, and the picture is completely changed. In words, for each     fixed $i_1<\cdots<i_p$, the event that $\bigcap_{r=1}^p\vv\tau\topp{i_r,d_{i_r}}\cap\{1,\dots,n\}\ne\emptyset$ becomes a rare one for $n$ large, and essentially we shall identify an appropriately chosen region $\calD^* \subset \{(i_1,\dots,i_p)\in\N^p: i_1<\cdots<i_p\}$, so that on this region a Poisson limit theorem holds for the total number of occurrence of a large collection of rare events (in the exact rare events to be considered, additionally the magnitude of $[\varepsilon_\vvi][\eta_\vvi]$ shall exceed an properly chosen growing threshold, which we omit in this remark). {\em In words, the formulation of extreme values in this regime is due to the occurrence of rare events, and any single collection of $p$ renewals indexed by $i_1<\cdots<i_p$ does not have contribution to the limit.}
\end{Rem}
\begin{Rem}
Our proof for the critical and sub-critical regimes is based on the two-moment method for Poisson approximation \citet{arratia89two}. This step  is the most involved argument in our paper, due to the underlying delicate mechanism of the formulation (more precisely, the region $\calD^*$ in the previous remark) of rare events.  First, there is another phase transition within the sub-critical regime: there is another sub-regime so that one can apply a straightforward application of the two-moment method, and yet for the rest one has to truncate the process first carefully to avoid certain local dependence that may block the method (see Remark \ref{rem:region} for a more detailed explanation; at a high level, the delicacy is similar to extremes for random energy model as explained in \citep{kistler15derrida}, a classical example for extremes of Gaussian random variables with strong dependence). Second, the critical regime also differs slightly from the sub-critical regime  (see Remark \ref{rem:difference}).

\end{Rem}
\begin{Rem}Our results,
 actually, suggest a full point-process convergence of
\[
\summ k1n\ddelta{X_k/c_n,k/n},
\]
as in \citep{basrak09regularly}.  
Such a full point-process convergence has many other consequences, including the convergence of the random sup-measures as in Theorem \ref{thm:RSM}. Another notable consequence would be a functional central limit theorem for the partial-sum process. 

However, the widely applied classical method for proving such a full point-process convergence \citep{davis95point,davis98sample,basrak09regularly,kulik20heavy}   does not apply here. This is closely related to the fact that for our model at sub-critical regime, the candidate extremal index $\vartheta = \mathfrak q_{F,p}$ and extremal index $\theta = \mathfrak q_{F,p}\mathsf D_{\beta,p}$  do not equal, and yet the classical approach requires necessarily $\vartheta = \theta$. The fact and related background on why/how the two indices are not the same is explained in full details in the accompanying paper 
\citep[Section 1.4]{bai21tail}.

We expect the two-moment method  can be adapted to prove the full point-process 
convergence,
 and this will be addressed in another paper. 
\end{Rem}
{\em The paper is organized as follows.} Section \ref{sec:background} provides related backgrounds, notably on multiple-stable processes and renewal processes. For Theorem \ref{thm:RSM}, the super-critical, sub-critical and critical regimes are proved in Sections \ref{sec:RSM}, \ref{sec:sub-critical} and \ref{sec:critical}, respectively.

\section{Preliminary results}\label{sec:background}
\subsection{Intersections of renewal processes with infinite mean}
Throughout, our references on discrete-time renewal processes are  \citet[Appendix A.5]{giacomin07random} and \citet[Section 8.7.1]{bingham87regular}.
Besides the notations and properties of renewal processes in Section \ref{sec:1}, we shall also use the renewal mass function of a standard renewal process (no shift) with renewal distribution $F$:
\[
u(k) := \proba(k\in \vv\tau),\ k\in\N_0.
\]
Note that $u(0) = 1$. It is well-known (e.g., \cite[Theorem 8.7.3]{bingham87regular}) that the relation 
\equh\label{eq:u(n)}
u(n)\sim \frac{n^{\beta-1}}{\mathsf C_F\Gamma(\beta)\Gamma(1-\beta)},  \mmas n\to\infty,
\eque
implies 
the assumption \eqref{eq:F}, and furthermore 
under the assumption \eqref{eq:Doney} the two are equivalent (cf.~\citep{doney97onesided}). 

Next, we recall some properties of the intersected   process 
of $p$ i.i.d.~copies of $\vv\tau$,  
$\vv\eta 
:=
 \bigcap_{r=1}^p\vv\tau\topp r = \{\eta_0=0,\eta_1,\eta_2,\dots\}$.    This is again a (non-shifted) renewal process, although  when $\beta_p<0$, it is  terminating, namely, $\eta_1 = \infty$ with strictly positive probability
\equh\label{eq:c_Fp}
\mathfrak q_{F,p} = \proba(\eta_1 = \infty) = \limn \wb F_p(n)\in(0,1).
\eque
Here and below  $\wb F_p(x) = 1-F_p(x)$ with $F_p$ denoting the cumulative distribution function of $\eta_1$.  Recall the renewal process $\vv\eta $ is said to be null-recurrent, if  $\proba(\eta_1<\infty)=1$ and $\esp \eta_1=\infty$.

\begin{Lem}
If $\beta_p\ge 0$, then the renewal process $\vv\eta$ is null-recurrent and 
\[
\wb F_p(n) \sim \begin{cases}
\displaystyle n^{-\beta_p}\frac{(\mathsf C_F\Gamma(\beta)\Gamma(1-\beta))^p}{\Gamma(\beta_p)(\Gamma(1-\beta_p))}, & \mbox{ if } \beta_p>0, \\\\
\displaystyle
\frac{(\mathsf C_F\Gamma(\beta)\Gamma(1-\beta))^p}{\log n}, & \mbox{ if } \beta_p = 0.
\end{cases}
\]
If $\beta_p<0$, then the renewal process $\vv\eta$ is terminating with   $P(\eta_1=\infty)=\mathfrak q_{F,p}$ in \eqref{eq:c_Fp}.
\end{Lem}
\begin{proof}
Let $u_p(n)$ denote the renewal mass function of $\vv\eta$. Clearly by independence, 
\equh\label{eq:u_p}
u_p(n) = u(n)^p\sim \frac{n^{\beta_p-1}}{(\mathsf C_F\Gamma(\beta)\Gamma(1-\beta))^p},  \mmas n\to\infty,
\eque
regardless of the value of $\beta_p$ (and $u_p(0) = 1$ always). 
In the case $\beta_p> 0$, and necessarily $\beta_p<1$, \eqref{eq:u_p} implies  
\[
\wb F_p(n)\sim n^{-\beta_p}\frac{\mathsf (\mathsf C_F\Gamma(\beta)\Gamma(1-\beta))^p}{\Gamma(\beta_p)(\Gamma(1-\beta_p))}.
\]
This follows again from \cite[Theorem 8.7.3]{bingham87regular} and  was already used in \citep[Lemma A.1]{samorodnitsky19extremal}. 
The critical regime $\beta_p = 0$ is similar. In this case, 
\[
\summ k1n u_p(k)\sim \frac1{(\mathsf C_F\Gamma( \beta)\Gamma(1-\beta))^p}\log n, 
\]
and hence by \citep[Theorem 8.7.3]{bingham87regular}, it follows that
\equh\label{eq:wb F_p crit}
\wb F_p(n) \sim \frac{(\mathsf C_F\Gamma(\beta)\Gamma(1-\beta))^p}{\log n}.
\eque
The null-recurrence in these two cases readily follows.

In the case $\beta_p<0$, consider
\[
G =  \sif n0\inddd{n\in\vv\eta}.
\]
The conclusion follows since 
\[
\esp G = \sif n0u_p(n) = \sif n0 u(n)^p<\infty.
\]
(In fact, by the renewal property, $G$ is a geometric random variable    satisfying  $\proba(G = k) = \mathfrak q_{F,p}(1-\mathfrak q_{F,p})^{k-1}, k\in\N$, so that  $\esp G = 1/\mathfrak q_{F,p}$.) 
\end{proof}

\subsection{A representation of multiple-stable processes}
We shall work with the following equivalent representation of finite-dimensional distributions of $\{X_k\}_{k=1,\dots,n}$ in \eqref{eq:series infty p}:
\equh\label{eq:p>=1}
\ccbb{X_{k}}_{k=1,\dots,n}\eqd \ccbb{X_{n,k}}_{k=1,\dots,n}:= \ccbb{w_n^{p/\alpha}\sum_{0<i_1<\cdots<i_p}\frac{[\varepsilon_\vvi]}{[\Gamma_\vvi]^{1/\alpha}}\inddd{k\in \bigcap_{k=1}^p R_{n,i_k}}}_{k=1,\dots,n},
\eque
where $w_n$ is as in \eqref{eq:w_n}, $\{\varepsilon_i\}_{i\in\N}$ are i.i.d.~Rademacher random variables, $\{\Gamma_i\}_{i\in\N}$ are consecutive ordered points from a standard Poisson process on $[0,\infty)$, 
and $\{R_{n,i}\}_{i\in\N}$ are i.i.d.~copies of certain random closed set $R_n$, all three families are independent.
Here $R_n$ is described as follows. 
Suppose $\vv\tau^*$ is a shifted renewal process with the stationary shift distribution $\pi$ and renewal distribution $F$ defined on a certain measurable space  with respect to certain infinite measure $\mu^*$ (since $\pi$ is an infinite measure). Then, one can introduce a probability measure $\mu_n$ on the same measurable space via
\[
\frac{d\mu_n}{d\mu^*} = \frac{\inddd{\vv\tau^*\cap\{1,\dots,n\} \ne\emptyset}}{\mu^*(\{\vv\tau^*:\vv\tau^*\cap\{1,\dots,n\}\ne\emptyset\})} = \frac{\inddd{\vv\tau^*\cap\{1,\dots,n\} \ne\emptyset}}{w_n}.
\]
Then, the law of $R_n$ is the one induced by $\vv\tau^*$ with respect to the probability measure $\mu_n$. 
Moreover, it is immediately verified that
 \begin{enumerate}[(i)]
 \item $\proba(k\in R_n) = 1/w_n, k=1,\dots,n$ (shift invariance).
 \item $\proba(\min(R_n\cap\{k+1,k+2,\ldots\} ) \le k+ j\mid k\in R_n) = F(j)$ (Markov/renewal property). 
 \end{enumerate}
 We refer to \citep{owada15functional,samorodnitsky19extremal} for details. 
Note that we write $\{X_{n,k}\}_{k=1,\dots,n}$ to emphasize that the last expression of \eqref{eq:p>=1}  depends on $n\in\N$, although they correspond consistently to the same process $\{X_k\}_{k\in\N}$.

\section{Convergence for random sup-measures,  $\beta_p>0$}\label{sec:RSM}

Throughout $C$   denotes a generic positive constant whose value may change from line to line.

\subsection{Alternative representations of random sup-measures}\label{sec:agg clusters}
In this subsection, we first provide an alternative representation of the limit random sup-measure $\calM_{\alpha,\beta,p}$. This alternative representation will  not actually  be used in the proof of the limit theorem. It is, however, instructive for  understanding the arise of $\calM_{\alpha,\beta,p}$ in the limit theorem from discrete model $\{X_k\}_{k\in\N}$, which is  also related to the original definition of the limit random sup-measure in \cite{samorodnitsky19extremal} when $p=1$.
Introduce the following set-indexed process: for $G\in \cal{G}$,
\equh\label{eq:RSM USC}
\calM'_{\alpha,\beta,p}(G) :=\sup_{t\in G}\zeta_{\alpha,\beta,p}(t) \qmwith \zeta_{\alpha,\beta,p}(t) :=  \sum_{0<i_1<\cdots<i_p}\frac{[\varepsilon_\vvi]}{[\Gamma_\vvi]^{1/\alpha}}\inddd{t\in\bigcap_{r=1}^p\calR_{\beta,i_r}},\ t\in[0,1] ,
\eque
where the  consecutive  Poisson points $\{\Gamma_i\}_{i\in\N}$,  i.i.d. Rademacher sequences
$\{\varepsilon_i\}_{i\in\N}$,
 and i.i.d.\ shifted stable regenerative sets $\{\calR_{\beta,i}\}_{i\in\N}$ are as described in  Section \ref{sec:stable reg intro}.  The function $ \zeta_{\alpha,\beta,p}(t)$ may be viewed as a continuous-time analog of the representation    \eqref{eq:p>=1} (or \eqref{eq:series infty}) for $\{X_k\}_{k\in\N}$. Note that the measurability of $\calM_{\alpha,\beta,p}'(G)$ is not immediately clear by the definition. 
 
\begin{Prop}\label{pro:RSM}~
\begin{enumerate}[(i)]
\item For  $\calM_{\alpha,\beta,p}$ defined in \eqref{eq:RSM}, we have for every $G\in\calG$, almost surely
\[
\calM_{\alpha,\beta,p}(G)>0.
\]
\item Assume in addition that the underlying probability space is complete. Then for  $\calM'_{\alpha,\beta,p}$ defined in \eqref{eq:RSM USC},
we have for every $G\in \calG$, almost surely  
 \[
\calM_{\alpha,\beta,p}(G) = \calM'_{\alpha,\beta,p}(G). 
 \]
 \end{enumerate}
\end{Prop}
\begin{proof}[Proof of Proposition \ref{pro:RSM}] 
We first prove the first part.  The case of $G=\emptyset$ is trivial since both sides of \eqref{eq:RSM} are $-\infty$. 
From now on, fix an nonempty $G\in \calG$.  
 First, we show that $\calM_{\alpha,\beta,p}(G)>  0$ almost surely. To see this, note that 
\begin{align*}
&\proba\left( \sup_{J\in\calJ_{p,p'}}\calM_{\alpha,\beta,J}(G) \le 0\right)\le\proba\left(\calM_{\alpha,\beta,J}(G) \le 0, \ J=\{(1,\ldots,p)\},\{(p+1,\ldots,2p)\},\ldots \right)\\
= & \lim_{n\rightarrow\infty}\left[\proba\left(\calR_{\beta,{(1,\ldots,p)}}\cap G\ne\emptyset\right)\proba([\varepsilon_{(1,\ldots,p)}]=-1) +\proba\left(\calR_{\beta,{(1,\ldots,p)}}\cap G=\emptyset\right)\right]^n.
\end{align*} 
where we used independence in the last equality. The   limit above is zero since $\proba([\varepsilon_{(1,\ldots,p)}]=-1)=1/2$ and $\proba\left(\calR_{\beta,{(1,\ldots,p)}}\cap G\neq \emptyset\right)\ge \proba\left(B_{1-\beta_p,1} \in G\right)>0$, recalling $\calR_{\beta,{(1,\ldots,p)}}  \eqd \wt\calR_{\beta_p}+B_{1-\beta_p,1}$ with $\wt\calR_{\beta_p}$ and $B_{1-\beta_p,1}$   as described in Section \ref{sec:stable reg intro}.

Now we prove the second part. We start by showing that on an event with probability one,   $\calM'_{\alpha,\beta,p}(G)\le \calM_{\alpha,\beta,p}(G)$. (Note that a priori the measurability of $\calM_{\alpha,\beta,p}'$ is not clear yet.)
Introduce the random index set $J_t:=\{\vvi\in \calD_p: t\in \calR_{\beta,\vvi}\}$. Then for $t\in G$, with the notation in \eqref{eq:aggregated cluster p=1}, we have   
\[
 \zeta_{\alpha,\beta,p}(t)= \sum_{\vvi \in J_t}  
 \frac{[\varepsilon_\vvi]}{[\Gamma_\vvi]^{1/\alpha}}= \sum_{\vvi \in J_t}  
 \frac{[\varepsilon_\vvi]}{[\Gamma_\vvi]^{1/\alpha}} \inddd{\calR_{\beta,J_t}\cap G\ne\emptyset}=\calM_{\alpha,\beta,J_t}(G),
\]
where we understand a summation over $J_t$ as $0$ if $J_t=\emptyset$.
Then 
\equh\label{eq:J_t in Jpp'}
\proba(\Omega_0) = 1 \qmwith \Omega_0 := \ccbb{\calR_{\beta,J} = \emptyset \mfa J\subset\N, |J|>p'},
\eque by the choice of $p'$ in \eqref{eq:p'}, and it follows that for all $\omega\in\Omega_0$, $J_t(\omega)\in\calJ_{p,p'}$ for all $t\in G$ such that $J_t \ne\emptyset$. Also, let $\Omega_1 = \{\calM_{\alpha,\beta,p}(G)>0\}$. So we have 
\[
\calM_{\alpha,\beta,p}'(G) = \sup_{t\in G}\zeta_{\alpha,\beta,p}(t) = \pp{\sup_{t\in G: J_t\ne\emptyset}\calM_{\alpha,\beta,J_t}(G)}_+\le \calM_{\alpha,\beta,p}(G) \quad \mbox{ on $\Omega_0\cap\Omega_1$.}
\] 

Next, we show that $\calM_{\alpha,\beta,p}'(G)\ge \calM_{\alpha,\beta,p}(G)$ on the event of probability one.   First, since $[\Gamma_\vvi]^{-1/\alpha}\rightarrow 0$ when $\max(\vvi)\rightarrow \infty$ almost surely, it 
follows
that the supremum over $J\in\calJ_{p,p'}$ in $\calM_{\alpha,\beta,p}(G)$ is attainable almost surely. Let $\Omega_2$ denote this event. So on $\Omega_0\cap\Omega_1\cap\Omega_2$,
\[
\calM_{\alpha,\beta,p}(G)=  \calM_{\alpha,\beta,J(\what \vvc)}(G)>0
\]
 for some random $\what \vvc\in \calD_q$  with random  $q\in\{p,p+1,\dots,p'\}$, and in particular $\calR_{\beta,J(\what\vvc)}\cap G\ne\emptyset$.
It suffices to show that 
\[J(\what\vvc)=J_t\text{ for some }t\in \calR_{\beta,J(\what\vvc)}\cap G \quad \mbox{ on $\Omega_0\cap\Omega_1\cap\Omega_2$.}
\]  
(In words, there is a location $t\in \calR_{\beta,J(\what\vvc)}\cap G$   exactly covered by the aggregated cluster corresponding to $J(\what{\vv c})$.) 
Note that on the event $\Omega_0\cap\Omega_1\cap\Omega_2$, for all $t\in \calR_{\beta,J(\what\vvc)}\cap G$, necessarily $ J(\what\vvc)\subset  J_t$.  It remains to show the inclusion is in fact an equality on the event mentioned above.   We now prove   by contradiction: suppose on a sub-event say $E$ of $\Omega_0\cap\Omega_1\cap\Omega_2$ with  strictly positive probability, we have $J(\what\vvc)\subsetneq J_t$ for all $t\in\calR_{\beta,J(\what\vvc)}\cap G$  (in words, every $t$ is covered by some aggregated cluster indexed by $J(\vv c)$ such that $\vv c\supsetneq \what{\vv c}$). Recall that in view of \eqref{eq:J_t in Jpp'},  $q\le p'-1$ on the event $E$. It then
 follows   that on the event $E$, 
\[
\emptyset\neq\calR_{\beta,J(\what\vvc)}\cap G \subset \left(\bigcup_{ t\in\calR_{\beta,J(\what\vvc)}\cap G} \calR_{\beta,J_t}\right)\cap G \subset \pp{
 \bigcup_{\substack{\vvc'\in \calD_{q+1}\\\what\vvc\subset\vvc'}} 
\calR_{\beta,J(\vvc')}}\cap G.
\]
 But 
 almost everywhere on the sub-event of $E$,  $\calR_{\beta,J(\what\vvc)}\cap G$  has Hausdorff dimension $\beta_q$, while each   $\calR_{\beta,J(\vvc')}$ in the last expression above with $\vvc'\in\calD_{q+1}$ has Hausdorff dimension $\beta_{q+1}<\beta_q$ and hence the countable union  in the last expression  has Hausdorff dimension $\beta_{q+1}<\beta_q$   in view of the countable stability property   (cf.\ \cite[Section 3.2]{falconer14fractal}).  So the event $E$ cannot have strictly positive probability. 
 
In summary, we have proved that on an event of probability one $\calM_{\alpha,\beta,p}(G) = \calM'_{\alpha,\beta,p}(G)$. By completeness of probability space, $\calM'_{\alpha,\beta,p}(G)$ is a random variable. This completes the proof.
\end{proof}

\subsection{Proof of Theorem \ref{thm:RSM}, $\beta_p>0$}
 Recall the model in  \eqref{eq:series infty p} and consider the associated empirical random sup-measure $M_n$ as  in \eqref{eq:fdd RSM}.
The goal of this section is to prove the following case of Theorem \ref{thm:RSM} (recall the  characterization of the weak convergence in the 
space of sup-measures  discussed around \eqref{eq:fdd RSM conv}), $\beta_p>0$: 
\equh\label{eq:thm super-crit}
\frac1{w_n^{p/\alpha}}\ccbb{M_n(I)}_{I\in\calG_0}\fddto
\ccbb{\calM_{\alpha,\beta,p}(I)}_{I\in\calG_0} \eqd\ccbb{\sup_{J\subset\calJ_{p,p'}}\calM_{\alpha,\beta,J}(I)}_{I\in\calG_0}.
\eque
In view of \eqref{eq:w_n}, in Theorem \ref{thm:RSM}, 
\equh\label{eq:C super-crit}
\mathfrak C_{F,p} = \pp{\frac{\mathsf C_F}{1-\beta}}^p. \eque
 We start with a truncation argument.  
 Define first 
\[
R_{n,\vvi} 
:=
 \bigcap_{k=1}^p R_{n,i_k} , \ \vvi=(i_1,\ldots,i_p)\in \calD_p.
\]
Recall that $M_n(G) = \max_{i/n\in G}X_i$ with $X_i$ in \eqref{eq:p>=1}.
 For each $\ell\in\N$ and $G\in \calG$, consider
\begin{equation}\label{eq:Mn truncate}
M_{n,\ell}(G)  := \max_{k/n\in G} w_n^{p/\alpha} \sum_{\vvi\in\calD_p, i_p\le \ell} \frac{[\varepsilon_\vvi]}{[\Gamma_\vvi]^{1/\alpha}} \inddd{k\in R_{n,\vvi}}.
\end{equation}
\begin{Lem}\label{lem:truncation}
For any $\epsilon>0$, we have
\[
\lim_{\ell\to\infty} \limsupn \proba\left(\max_{k=1,\ldots,n}\abs{ \sum_{\vvi\in\calD:i_p>\ell} \frac{[\varepsilon_\vvi]}{ [\Gamma_\vvi]^{1/\alpha}} \inddd{k\in R_{n,\vvi}}} >\epsilon\right)= 0.
\]
\end{Lem}
\begin{proof}
The argument  is similar to those in \citep{bai20functional,samorodnitsky19extremal}.
Fix finite $M\ge 2p/\alpha$.
It suffices to show for any $0\le  q \le p-1$ and fixed $\vvi = (i_1,\dots,i_q)\in\calD_q$ with $i_q\le M$  that  
\[
\lim_{\ell\to\infty} \limsupn \proba\left(\max_{k=1,\ldots,n}\left|  \frac{[\varepsilon_{i_{1:q}}]}{[\Gamma_{i_{1:q}}]^{1/\alpha}} \sum_{M<i_{q+1}<\cdots<i_p, i_p>\ell}   \frac{[\varepsilon_{i_{q+1:p}}]}{[\Gamma_{i_{q+1:p}}]^{1/\alpha}} \inddd{k\in R_{n,\vvi}}\right| >\epsilon\right)= 0.
\]
Here and below we write $\varepsilon_{i_{a:b}} = (\varepsilon_{i_a},\varepsilon_{i_{a+1}},\dots,\varepsilon_{i_b})$ and $[\varepsilon_{i_{a:b}}] = \prod_{r=a}^b\varepsilon_{i_r}$.  By Markov inequality, union bound and stationarity in $k$, the probability above is bounded by
\[
\frac{n}{\epsilon^r} \esp  \left| \frac{ [\varepsilon_{i_{1:q}}]}{ [\Gamma_{i_{1:q}}]^{1/\alpha}} \sum_{M<i_{q+1}<\cdots<i_p, i_p>\ell}   \frac{[\varepsilon_{i_{q+1:p}}]}{ [\Gamma_{i_{q+1:p}}]^{1/\alpha}} \inddd{1\in R_{n,\vvi}}\right|^r 
\]
for some $r>0$ satisfying $1/r=1/r'+1/\gamma$, where $\gamma$ is chosen large enough so that
\[
\gamma>\frac{2}{1-\beta},
\]  and $r'$ is chosen so that $(p-1)r'/\alpha<1$ which ensures $\esp([ \Gamma_{i_{1:q}}]^{-r'/\alpha})<\infty$ for all $q\le p-1$ (cf.\ \citep[eq.\ (3.2)]{samorodnitsky89asymptotic}).
 By H\"older's inequality,  it remains to show 
\[
\lim_{\ell\to\infty} \limsupn n \esp  \left|    \sum_{M<i_{q+1}<\cdots<i_p, i_p>\ell}\frac{  [\varepsilon_{i_{q+1:p}}]}{ [\Gamma_{i_{q+1:p}}]^{1/\alpha}} \inddd{1\in R_{n,\vvi}}\right|^\gamma =0.
\] 
By a generalized Khinchine inequality \citep[Theorem 1.3(ii)]{samorodnitsky89asymptotic}, exploring the orthogonality induced by $[\varepsilon_{i_{q+1:p}}]$, and the fact $\proba(1\in R_{n,\vvi})= \proba(1\in R_{n,1})^{p-q}=w_n^{-(p-q)}$, the expectation in the displayed line above is bounded from above by,  up to a multiplicative constant,
\[
\left(\esp  \left|    \sum_{M<i_{q+1}<\cdots<i_p, i_p>\ell}   \frac{[\varepsilon_{i_{q+1:p}}]}{ [\Gamma_{i_{q+1:p}}]^{1/\alpha} }\inddd{1\in R_{n,\vvi}}\right|^2 \right)^{\gamma/2}
\le   w_n^{-\gamma (p-q)/2} \left( \sum_{M<i_{q+1}<\cdots<i_p, i_p>\ell}   \esp [\Gamma_{i_{q+1:p}}]^{-2/\alpha}\right)^{\gamma/2}.
\]
By the choice of $\gamma$, we have $nw_n^{-\gamma(p-q)/2}\le nw_n^{-\gamma/2}\rightarrow 0$. To conclude the proof, we apply the bound 
\begin{equation}\label{eq:esp Gamma product bound}
\esp [\Gamma_{i_{q+1:p}}]^{-2/\alpha}\le C [i_{q+1:p}]^{-2/\alpha},
\end{equation} 
which holds since $i_{q+1}>M\ge 2 p/\alpha\ge 2(p-q)/\alpha$  \citep[eq.\ (3.2)]{samorodnitsky89asymptotic}, so that the last multiple sum above is bounded.
 \end{proof}
Next, we define for any $J = \{\vvi_1,\dots,\vvi_r\}$, $\vv i_s=(i_{s,1},\ldots,i_{s,p})$, $s=1,\ldots,r$, that
\[
\cup J :=\ccbb{i_{s,h}:  s=1,\dots,r,\ h=1,\dots,p}\subset\N 
\]
and introduce of a finite truncation of $\calJ_{p,p'}$ in \eqref{eq:Jpp} as
\[
\calJ_{p,p'}(\ell)  := \ccbb{J\in \calJ_{p,p'}: \max(\cup J) \le \ell}.
\]

Next, as a discrete analog of Proposition \ref{pro:RSM}, we approximate $M_{n,\ell}(G)$ in \eqref{eq:Mn truncate} by 
\begin{equation}\label{eq:tilde Mnl}
\wt M_{n,\ell}(G):=\max_{J\subset \calJ_{p,p'}(\ell)}\wt M_{n,J}(G),
\end{equation}
where
\begin{equation}\label{eq:tMnJ}
\wt M_{n,J}(G)  := \begin{cases}
\displaystyle w_n^{p/\alpha}\sum_{\vvi\in J}\frac{[\varepsilon_\vvi]}{[\Gamma_\vvi]^{1/\alpha}}, & \displaystyle\mbox{ if } \frac1n R_{n,J}\cap G\ne\emptyset,
\\
-\infty, & \mbox{ otherwise.}
\end{cases}
\qmwith R_{n,J} := \bigcap_{\vvi\in J}R_{n,\vvi}.
\end{equation}
 
\begin{Lem}\label{lem:sparse}
For every $G\in\calG$ fixed, 
\[
\limn \proba\pp{M_{n,\ell}(G) = \wt M_{n,\ell}(G)} = 1.
\]
\end{Lem}
\begin{proof}
Suppose $G$ is nonempty; otherwise the result is trivial.
Consider the following event:
\[
\Omega_{n,1}(\ell):=\ccbb{\max_{k=1,\dots,n}\summ i1\ell\inddd{k\in R_{n,i}}\le p'}.
\]
Restricted to $\Omega_{n,1}(\ell)$, one readily checks that $M_{n,\ell}(G)\le \wt M_{n,\ell}(G)$. For the other direction, introduce further
\[
\Omega_{n,2}(J):=\ccbb{R_{n,J}\cap G = \emptyset}\cup\ccbb{\exists k\in\{1,\dots,n\} \mbox{ s.t. } k\in R_{n,J}, k\notin \bigcup_{i\in\{1,\dots,\ell\}\setminus(\cup J)}R_{n,i}},
\]
and
\[
\Omega_{n,2}(\ell) :=\bigcap_{J\in \calJ_{p,p'}(\ell)}\Omega_{n,2}(J).
\]
Restricted to $\Omega_{n,2}(\ell)$, one has  $\wt M_{n,\ell}(G)\le M_{n,\ell}(G)$ 
because the $J$ set which maximizes \eqref{eq:tilde Mnl} can always be realized by   $\{\vv i=(i_1,\ldots,i_p) \in \calD_p,\ i_p\le \ell:\  k\in R_{n,\vvi} \}$ for some $k$. Therefore, it remains to prove that $\limn\proba(\Omega_{n,i}(\ell)) = 1$ for $i=1,2$. For $\Omega_{n,1}(\ell)$, this follows from the fact that $R_{n,\vvi}\weakto\emptyset$ (e.g.~\citep[Theorem 3.1]{samorodnitsky19extremal}) for all $\vvi\in\calD_{p'+1}$. The fact that $\limn\proba(\Omega_{n,2}(\ell)) = 1$ has played a crucial role in \citep[Lemma 5.5]{samorodnitsky19extremal} (see also \citep[Lemma 3.7]{wang22choquet}). 
\end{proof}

\begin{proof}[Proof of Theorem \ref{thm:RSM}, $\beta_p>0$]
The goal is to show   \eqref{eq:thm super-crit}. 
Recall $M_{n,\ell}$ in \eqref{eq:Mn truncate}  and $\wt M_{n,\ell}$ in \eqref{eq:tilde Mnl}, which serve  as approximations for $M_n$. 
Introduce the truncated version of $\calM_{\alpha,\beta,p} $ as
\[
\calM_{\alpha,\beta,p}\topp\ell(G):=\max_{J\in\calJ_{p,p'}(\ell)}\calM_{\alpha,\beta,J}(G),\quad G\in \calG.
\]
We first show  
\equh\label{eq:ell convergence}
\ccbb{\frac1{w_n^{p/\alpha}} \wt M_{n,\ell}(I)}_{I\in\calG_0}\fddto\ccbb{\calM_{\alpha,\beta,p}\topp\ell(I)}_{I\in\calG_0}.
\eque
Indeed, the above follows if for all   $I_s\in\calG_0, s=1,\dots,d$,  $d\in \N$,   
\equh\label{eq:ell convergence1}
\ccbb{\frac1{w_n^{p/\alpha}}\wt M_{n,J}(I_s)}_{J\in\calJ_{p,p'}(\ell),\,  s=1,\dots,d}\weakto \ccbb{\calM_{\alpha,\beta,J}(I_s)}_{J\in\calJ_{p,p'}(\ell),\, s=1,\dots,d}. 
\eque
Comparing the two sides (see \eqref{eq:tMnJ} and \eqref{eq:RSM}), it suffices to show
\[
\ccbb{\frac1nR_{n,\vvi}}_{\substack{\vvi = (i_1,\dots,i_q)\in \calD_q\\
 i_q\le\ell, q=p,\dots,p'}}
 \weakto\ccbb{\calR_{\beta,\vvi}}_{\substack{\vvi = (i_1,\dots,i_q)\in \calD_q\\
 i_q\le\ell, q=p,\dots,p'}},
\]
 a joint convergence in distribution of finitely many random closed sets. This joint convergence, a non-trivial result since  $R_{n,\vvi}$ and $R_{n,\vvi'}$ are dependent when $\vvi\cap\vvi'\ne\emptyset$, was established in \citep[Theorem 5.4]{samorodnitsky19extremal}. The desired \eqref{eq:ell convergence1} and hence \eqref{eq:ell convergence} now  follow. 
Next, for any   $I\in \cal{G}_0$, we have
\[
\lim_{\ell\to\infty}\limsupn\proba\pp{\frac1{w_n^{p/\alpha}}\abs{M_n(I)-\wt M_{n,\ell}(I)}\ge \epsilon} = 0.
\]
 Above $\wt M_{n,\ell}(I)$ may be $-\infty$ while $M_n(I)$ (and also $M_{n,\ell}(I)$ appearing below) is almost surely finite for all large enough $n$; we understand $\abs{M_n(I)-\wt M_{n,\ell}(I)}$ as $\infty$ when 
 $M_n(I)$ is finite and 
$\wt M_{n,\ell}(I)=-\infty$.  
Indeed, the above follows from Lemmas \ref{lem:truncation} and \ref{lem:sparse} applied to the two probabilities respectively on right-hand side of following inequality
\[
\proba\pp{\frac1{w_n^{p/\alpha}}\abs{M_n(I)-\wt M_{n,\ell}(I)}\ge \epsilon} 
\le 
\proba\pp{\frac1{w_n^{p/\alpha}}\abs{M_n(I)-M_{n,\ell}(I)}>\varepsilon}+\proba\pp{M_{n,\ell}(I)\ne\wt M_{n,\ell}(I)}.
\]
In addition, by  monotonicity, one has  
\[
\lim_{\ell\to\infty}\calM_{\alpha,\beta,p}\topp\ell(I)= \calM_{\alpha,\beta,p}(I).
\]
The desired result now follows from a routine triangular approximation argument \citep[Theorem 3.2]{billingsley99convergence}.
\end{proof}

\section{Convergence for random sup-measures,  $\beta_p<0$}\label{sec:sub-critical}
The goal of this section is to prove the following case of Theorem \ref{thm:RSM} (recall the  characterization of the weak convergence in the 
space of sup-measures
   discussed around \eqref{eq:fdd RSM conv}), with $\beta_p<0$:
\equh\label{eq:thm sub-crit}
\frac1{c_n}\ccbb{M_n(I)}_{I\in\calG_0}\fddto
\ccbb{\mathfrak C_{F,p}^{1/\alpha}\calM_{\alpha}^{\rm is}(I)}_{I\in\calG_0},
\eque
where we have
\equh
c_n = \pp{n\log ^{p-1}n}^{1/\alpha}
\qmand \label{eq:C_Fp}
\mathfrak C_{F,p} = \frac1{2p!(p-1)!} q_{F,p}\mathsf D_{\beta,p},
\eque   with
 $\mathfrak q_{F,p}$ as in \eqref{eq:c_Fp}, and
\begin{equation}\label{eq:shape parameter}
\mathsf D_{\beta,p} :=  \sum_{s=q_{\beta,p}}^p (-1)^{p-s}\binom ps(-\beta_s)^{p-1} \qmwith q_{\beta,p}:=\min\{q\in\N:\beta_q <0\}.
\end{equation}
It can be shown that  the parameter $\mathsf D_{\beta,p}$   takes value in $(0,1)$ when $\beta\in (0,1-1/p)$ (see \cite{bai21tail}).
Recall from \eqref{eq:w_n} that  
$w_n \sim  \mathsf C_F(1-\beta)\inv n^{1-\beta}$. Throughout this section we set 
\equh\label{eq:r_n sub}
 r_n=\frac{w_n^p}{c_n^{\alpha}} \sim \pp{\frac{\mathsf C_F}{1-\beta}}^p\frac{n^{-\beta_p}}{ \log^{p-1}n}.
\eque
To prove \eqref{eq:thm sub-crit}, the strategy is to first approximate the series representation \eqref{eq:p>=1}   by  a truncated version, and then to  establish  the point-process convergence of block maxima for the truncated process.
For each $n\in \N$,  we recall the series representation of finite-dimensional distribution $\{X_{n,k}\}_{k=1,\dots,n}\eqd\{X_k\}_{k=1,\dots,n}$ in \eqref{eq:p>=1}:  
Introduce, for a sequence $d_n\to\infty$ such that $k_n := \floor{n/d_n}\rightarrow\infty$,  
\[
\calI_{n,j}:=\{(j-1)d_n+1,\dots,jd_n\},\ j=1,\dots,k_n,
\]
which are $k_n$ non-overlapping blocks each of length $d_n$. 
Introduce   a sub-domain of $\calD_p$ with a product constraint as
\[
\calD^*_p(x):=\ccbb{\vv i=(i_1,\dots,i_p)\in\N^p:i_1<\cdots<i_p, [\vvi]= i_1\ldots i_p  \le x},\ x>0,
\]
as well as
\equh\label{eq:H(n,K)}
\calH(n,K) := \ccbb{\vvi\in \calD^*_p(Kr_n):i_p\le w_n},
\eque
where $r_n$ is as in \eqref{eq:r_n sub}.
Note that $\calH(n,K) = \calD_p^*(Kr_n)$ if $r_n\le K w_n$, which holds  for large $n$  if $\beta_{p-1}\ge 0$. It is crucial to work with $\calH(n,K)$ instead of $\calD_p^*(Kr_n)$ when $\beta_{p-1}<0$. See Remark \ref{rem:region} for explanations. 

Now we introduce the truncated process, for $K>0$,
\[
X_{n,k}\topp K := \sum_{\vvi\in\calH(n,K)}w_n^{p/\alpha}
\frac{[\varepsilon_\vvi]}{[\Gamma_\vvi]^{1/\alpha}}\inddd{k\in R_{n,\vvi}}, k=1,\dots,n, 
\]
and a point process involving block maxima:
\[
\xi_{n,K}  := \summ j1{k_n}\ddelta{\what m_{n,j}\topp K/c_n,j/k_n}\inddd{\what m_{n,j}\topp K> 0}
\qmwith \what m_{n,j}\topp K:=\max_{k\in\calI_{n,j}}X_{n,k}\topp K, j=1,\dots,k_n.
\]
It is worth mentioning that since we consider limit theorem for maximum, the negative values of $ \what m_{n,j}\topp K$ will play no role and hence are  excluded in the point process  $\xi_{n,K}$; see the relation \eqref{eq:wt xi''}  in Section \ref{sec:remainder} below. Establishing weak convergence of the point process  $\xi_{n,K}$ as $n\rightarrow\infty$ is a key step in the proof of \eqref{eq:thm sub-crit}. Throughout, we let $\mathfrak M_p((0,\infty]\times[0,1])$ denote the space of Radon point measures on $(0,\infty]\times[0,1] $ with the vague topology. The standard reference is \citet{resnick87extreme}.
\begin{Prop}\label{prop:1}
With $\beta_p<0$, $d_n\to\infty$ and $k_n\rightarrow\infty$ as $n\rightarrow\infty$, we have the following weak convergence
\equh\label{eq:PPP_limit'}
\xi_{n,K} \weakto  \xi_K:=\sif\ell1 \ddelta{\mathfrak C_{F,p}
^{1/\alpha}
\Gamma_\ell^{-1/\alpha},U_\ell}\inddd{\mathfrak C_{F,p} 
^{1/\alpha}
\Gamma_\ell^{-1/\alpha}>{K^{-1/\alpha}}},
\eque
in $\mathfrak M_p((0,\infty]\times [0,1])$, where $\mathfrak C_{F,p}$ is as in \eqref{eq:C_Fp}.
\end{Prop}
Note that the limit $\xi_K$ in \eqref{eq:PPP_limit'} is a Poisson point process.

In order to prove Proposition \ref{prop:1}, 
we shall work with the following point processes 
\begin{align*}
 \eta_{n,K}
&:= \summ j1{k_n}\sum_{\vvi\in \calH(n,K)}\ddelta{([\Gamma_\vvi]/r_n)^{-1/\alpha} ,j/k_n}\inddd{R_{n,\vvi}\cap\calI_{n,j}\ne\emptyset,\  [\varepsilon_\vvi]=1 },\nonumber\\
 \wb\eta_{n,K}&
:= \summ j1{k_n}\sum_{\vvi\in \calH(n,K)}\ddelta{ ([\vv i]/ r_n)^{-1/\alpha},j/k_n}\inddd{R_{n,\vvi}\cap \calI_{n,j}\ne\emptyset,\ [\varepsilon_\vvi]=1 }.
\end{align*}
Note that there is only  one difference between the two above: the random product $[\Gamma_\vvi]$  in $ \eta_{n,K}$ replaced by the non-random product $[\vvi]$ in $\wb\eta_{n,K}$.   
An overview of the approximations behind Proposition \ref{prop:1} is as follows. For some test set $E$,
\equh\label{eq:structure}
 \xi_{n,K}(E) \quad\underbrace{\approx\quad \eta_{n,K}(E)\quad \approx}_{\rm Lemma~\ref{lem:1}}\quad \wb\eta_{n,K}(E) \underbrace{\weakto}_{\rm Lemma~\ref{lem:2}}\xi_K(E).
\eque
The approximations above  are understood with letting $n\to\infty$, with $E$ fixed and $K$   picked large enough depending on $E$.
The meaning of the approximations `$\approx$' will be made precise in  the lemmas indicated.
In particular, we shall consider sets of the form
\equh\label{eq:E}
E := (y_0,y_1]\times \wt I   \subset(0,\infty] \times[0,1] \qmwith \wt I  = (a,b] \mbox{ or } [0,b]\subset[0,1],
\eque
for some $0< y_0<y_1\le\infty$, and each set satisfying 
\equh\label{eq:kappa0}
 0 <K^{-1/\alpha}<y_0.
\eque
We mention that for  proving  Proposition \ref{prop:1} we shall fix $K$ and consider all $E$ described above, whereas for  proving \eqref{eq:thm sub-crit}, we shall work with a fixed collection of $E_1,\dots,E_d$, and   let $K\to\infty$  so that \eqref{eq:kappa0} is eventually satisfied.

At last, say for  the marginal convergence in \eqref{eq:thm sub-crit} with $I\in\calG_0$, we have for $y>K^{-1/\alpha}$,
\equh\label{eq:structure1}
\proba\pp{M_n(I)\le y} \underbrace\approx_{\rm Lemma~\ref{lem:remainder}} \proba\pp{M_{n,K}(I)\le y} \approx
\proba\pp{ \xi_{n,K}( (y,\infty]\times I) = 0},
 \eque
where $M_{n,K}(I)$ is defined as $M_n(I)$ with  $X_k$ (or $X_{n,k}$)   replaced by the truncated version $X_{n,k}\topp K$.

The most involved step is a Poisson approximation between $\wb\eta_{n,K}$ and $\xi$ over certain test sets, which will be proved first in Section \ref{sec:two-moment}. We then prove  Proposition \ref{prop:1} in Section \ref{sec:approximation} and finally \eqref{eq:thm sub-crit} in Section \ref{sec:remainder}.

\subsection{Two-moment method for Poisson approximation}\label{sec:two-moment}

\begin{Lem}\label{lem:2}
For $K>0$ 
fixed and for all disjoint sets
 $E_s  =  (y_{s,0}, y_{s,1}]\times \wt I_s$,  
 $s=1,\dots,d$,  of the form \eqref{eq:E},  suppose 
\[
 0 <K^{-1/\alpha}<\min_{s=1,\dots,d} y_{s,0}.
\]
We have as $n\rightarrow\infty$,
\equh\label{eq:Poisson fdd}
(\wb\eta_{n,K}(E_1),\dots,\wb \eta_{n,K}(E_d))\weakto (\xi_K(E_1),\dots,\xi_K(E_d)).
\eque
Moreover, $\wb\eta_{n,k}\weakto \xi_K$ in $\mathfrak M_p((0,\infty]\times[0,1])$. 
\end{Lem}

Lemma \ref{lem:2} is a multivariate Poisson limit theorem for sums of dependent Bernoulli variables 
\[
\chi_{\vvi,j}:=  \inddd{R_{n,\vvi}\cap \calI_{n,j}\ne\emptyset,\ [\varepsilon_\vvi]=1},
\]
with summation index sets 
 \[\mathfrak I_{n,s} :=\ccbb{(\vvi,j)\in \calH(n,K)\times \{1,\ldots,k_n\}: (([\vvi]r_n)^{-1/\alpha},j/k_n)\in E_s},\ s=1,\dots.d,
 \]
 In this way, we write
 \[
 \wb \eta_{n,K}(E_s) = \sum_{(\vvi,j)\in\mathfrak I_{n,s}}\chi_{\vvi,j},\  s=1,\dots,d.
 \]
We apply the well-known  two-moment method of \citet[Theorem 2]{arratia89two}.  
The dependence structure of the Bernoulli random variables $\chi_{\vvi,j}$ can be described by a  dependency graph (\citep[Section 2.1]{penrose04random}), where disconnection of subsets of nodes implies independence. Obviously for  $\chi_{\vvi,j}$, we can define $(\vvi,j)\sim(\vvi',j')$, namely, these two nodes  to be connected,   if $\vvi\cap \vvi'\neq \emptyset$ (regardless of $j$ and $j'$). Here and below  $\vvi\cap\vvi'$ is  understood in the obvious way  by regarding $\vvi$ and $\vvi'$ as subsets of $\N$ and the 
elements
 of  $\vvi\cap\vvi'$ 
 listed
  in increasing order if it is nonempty.   
  
Before stating the conditions for the two-moment method, we introduce a few notations. First, introduce the following counting numbers (recall $\calH(n,K)$  in \eqref{eq:H(n,K)}):
\begin{align}
C_{n,1}(K) &:=\abs{\calH(n,K)},\label{eq:Cn1 def}\\
C_{n,2}(r,K) & :=  |\{(\vvi,\vvi'): \vvi,\vvi'\in \cal H(n,K),|\vvi\cap  \vvi'|=r\} |,  \ r=1,\dots,p,\label{eq:Cn2 def}
\end{align}
and the following probabilities 
\begin{align}
\rho_n & :=\esp\chi_{\vvi,j} = \proba([\varepsilon_\vvi] = 1)\cdot\proba(R_{n,\vvi}\cap \calI_{n,1}\neq \emptyset ) = \frac12\cdot \proba(R_{n,\vvi}\cap \calI_{n,1}\neq \emptyset )\label{eq:rho_n},\\
\rho_{n,\vvi,\vvi',j,j'} &  := \esp(\chi_{\vvi,j}\chi_{\vvi',j'}) \le \proba\pp{R_{n,\vvi}\cap\cal I_{n,j}\ne\emptyset, R_{n,\vvi'}\cap\calI_{n,j'}\ne\emptyset}.\label{eq:rho_n i j}
\end{align}
Note that $C_{n,2}(p,K) = C_{n,1}(K)$, and for all $j,j'$ fixed,  $\rho_{n,\vvi,\vvi',j,j'}$  has an identical value for all $\vvi,\vvi'$ such that $|\vvi\cap\vvi'| = r$. Note also that in \eqref{eq:rho_n i j} we bound $\proba([\varepsilon_\vvi] =1,  [\varepsilon_{\vvi'}] = 1) \le 1$.
Introduce
\begin{equation}\label{eq:rho_n(r)}
\rho_n(r) :=\begin{cases}
\displaystyle \sum_{j=1}^{k_n}\sum_{j'=1}^{k_n}\rho_{n,(1,\dots,p),(1,\dots,r,p+1,\dots,2p-r),j,j'}, & \mbox{ if } r=1,\dots,p-1,\\
\\
\displaystyle \sum_{j,j'\in\{1,\dots,k_n\},j\ne j'}\rho_{n,(1,\dots,p),(1,\dots,p),j,j'}, &\mbox{ if } r= p.
\end{cases}
\end{equation}
In this way, 
we have
\begin{equation}\label{eq:esp eta bar}
 \esp \wb\eta_{n,K}(E_s) = \sum_{(\vvi,j)\in \mathfrak I_{n,s}}\esp \chi_{\vvi,j} = |\mathfrak I_{n,s}|\cdot\rho_n,
 \quad  s=1,\ldots,d,
 \end{equation}
Next, we introduce, with $\mathfrak I_n:=\calH(n,K)\times\{1,\dots,k_n\}$, 
\begin{align}
b_{n,1}(K)&:=\sum_{(\vvi,j)\in\mathfrak I_n}\sum_{\substack{(\vvi',j')\in\mathfrak I_n\\ (\vvi,j)\sim(\vvi',j')}} \esp \chi_{\vvi,j}\esp\chi_{\vvi',j'} = k_n^2\sum_{\substack{\vvi,\vvi'\in\calH(n,K)\\ |\vvi\cap\vvi'|>0}}\rho_n^2 = k_n^2\rho_n^2{\summ r1{p}C_{n,2}(r,K)}\label{eq:b_n1}\\
b_{n,2}(K)&:=\sum_{(\vvi,j)\in\mathfrak I_n}\sum_{\substack{(\vvi,j)\ne(\vvi',j')\in\mathfrak I_n\\ (\vvi,j)\sim(\vvi',j')}} 
\esp (\chi_{\vvi,j}\chi_{\vvi',j'})=
\summ r1{p}\sum_{\substack{\vvi,\vvi'\in\calH(n,K)\\ |\vvi\cap\vvi'|=r}} \rho_n(r) = \summ r1{p}C_{n,2}(r,K)\rho_n(r),\label{eq:b_n2}
\end{align}
Then, by \citep[Theorem 2]{arratia89two}, the Poisson convergence \eqref{eq:Poisson fdd} follows once we show
\begin{equation}\label{eq:goal two moments}
\limn \esp \wb\eta_{n,K}(E_s) = \esp \xi_K(E_s), s=1,\dots,d,
 \qmand \limn (b_{n,1}(K)+b_{n,2}(K)) = 0.
\end{equation}

It remains to provide estimates for $C_{n,1}(K), C_{n,2}(r,K), \rho_n$ and $\rho_n(r)$, and we shall proceed one by one. Note that we need exact asymptotics of $C_{n,1}(K)$ and $\rho_n$, and only the orders of $C_{n,2}(r,K)$ and $\rho_n(r)$. 
\begin{Lem}\label{lem:D_p*}
As $r\rightarrow\infty$, we have 
\equh\label{eq:Dp}
 |\calD_p^*(r)|\sim \frac{r \log^{p-1}(r)}{p!(p-1)!}.
\eque
\end{Lem}
\begin{proof}
The key is to argue the integral approximation $|\calD_p^*(r)|  \sim \int_{[1 ,r]^p} 1_{\{x_1\cdots x_p\le r \}}d\vv x$ as $r\rightarrow\infty$.  Here we omit its proof, which  is similar to a more sophisticated case below in the proof of Lemma \ref{lem:Cn}. Then by change of variables,
\begin{align}
 \int_{[1 ,r]^p} 1_{\{x_1\cdots x_p\le r \}}d\vv x & =\int_{[0,\log(r)]^p} e^{y_1+\cdots+y_p} \inddd{y_1+\cdots+y_p\le \log(r)}d\vv y \nonumber\\
&=  \frac{1}{(p-1)!} \int_0^{\log(r)} z^{p-1} e^zdz\sim\frac{r\log^{p-1}(r)}{(p-1)!},\label{eq:simplex limit}
\end{align}
as $r\to\infty$.
\end{proof}

\begin{Lem}\label{lem:Cn}
Assume $\beta_p<0$. We have, for all $K>0$, 
\[
C_{n,1}(K) = |\calH(n,K)| \sim \frac{K}{p!(p-1)!} \mathsf D_{\beta,p}\cdot \frac{w_n^p}n,
\]
where
\[
\mathsf D_{\beta,p} :=  \sum_{s=q_{\beta,p}}^p (-1)^{p-s}\binom ps(-\beta_s)^{p-1} \qmwith q_{\beta,p}:=\min\{q\in\N:\beta_q <0\}.
\]
\end{Lem}
\begin{proof}
We shall approximate summations by the corresponding integrals as follows. For a general summable function  $f: \N^p\rightarrow \mathbb{R}$, $\sum_{\vvi \in \N^p } f(\vvi)=\int_{[0,\infty)^p} f(\ceil{x_1},\ldots \ceil{x_p} ) d\vv x =\int_{[0,\infty)^p} f(\floor{x_1}+1,\ldots \floor{x_p}+1) d\vv x$. Here and below $\floor{x}$ ($\ceil{x}$, resp.) denotes the greatest (smallest, resp.) integer less than or equal to (greater than or equal to, resp.) $x$, which should  be distinguished with $[ \vv x  ]$ which denotes the product of all components of a vector $\vv x$. 
We first derive a crude upper bound. Let $U_1,\ldots,U_p$ be i.i.d.\ uniform random variables  in $[0,1]$. Then
\begin{align}
C_{n,1}(K) & \le \frac{1}{p!} \int_{[0,\infty)^p} \inddd{\ceil{x_1} \cdots \ceil{x_p} \le K r_n, \ceil{x_i}\le w_n,i=1,\dots,p} d\vv x  \nonumber \\
&\le   
   \frac1{p!}\int_{\substack{x_1\cdots x_p\le K r_n\\ 0\le x_i\le w_n,i=1,\dots,p}} d\vv x \le \frac{w_n^p}{p!}\proba\pp{U_1\cdots U_p\le K r_n/w_n^p}\nonumber\\
& \sim \frac{w_n^p}{p!}\frac{K(r_n/w_n^p)(-\log(r_n/w_n^p)) ^{p-1}}{(p-1)!}\sim \frac {K\alpha^{p-1}}{p!(p-1)!}\cdot r_n \log^{p-1} (n),\label{eq:Cn1}
\end{align}
where we used the fact that $r_n/w_n^p = c_n^{-\alpha}\to 0$ as $n\to\infty$, and recalled in the third step that for $s\in(0,1)$,
\[
\proba(U_1\cdots U_p\le s) = s\summ k0{p-1}\frac{(-\log s)^k}{k!} \sim \frac {s(-\log s)^{p-1}}{(p-1)!}
\]
as $s\downarrow 0$ 
\citep[Lemma 3.1]{samorodnitsky89asymptotic}.   (This upper bound \eqref{eq:Cn1} is of the correct order $O(w_n^p/n)$, but not sharp in the multiplicative constant as shown below.)

Next we aim at  precise upper and lower bounds. Write 
$C_{n,1}(K)= A_n+B_n$
with
\begin{align*}
A_n &:=  |\ccbb{\vvi\in \N^p:  [\vvi]\le Kr_n,   2 \le i_1<\ldots<i_p\le w_n}|,\\
B_n &:= |\ccbb{\vvi\in \N^p:[\vvi] \le Kr_n, i_1=1, i_2<\ldots<i_p\le w_n}|.
\end{align*}
The crude upper bound \eqref{eq:Cn1} implies that
$B_n=O(r_n ( \log (n))^{p-2})$,
which will be eventually negligible.
On the other hand,
\begin{align*}
p!A_n&=|\ccbb{\vvi\in \N^p:[\vvi] \le Kr_n, 2\le  i_s\le w_n, s=1,\ldots,p, \ i_{s}\neq i_{t} \text{ if } s\neq t}|\\
 &= G_n -O(r_n ( \log (n))^{p-2}), 
\end{align*}
where
\[
G_n= |\ccbb{\vvi\in \N^p:[\vvi] \le Kr_n, 2\le i_s\le w_n, s=1,\ldots,p}|,
\] 
and the negligible $O(r_n ( \log (n))^{p-2})$ term follows from the upper bound \eqref{eq:Cn1} with $p$ replaced by $p-1$,$p-2$,\ldots,$1$, which   correspond to the number of distinct $i_s$'s. 

Now we provide precise upper and lower bounds for $G_n$:
\[
G_n = \int_{[1,\infty)^p} \inddd{ \ceil{x_1}  \cdots  \ceil{x_p} \le K r_n, \ceil{ x_i} \le w_n,i=1,\dots,p} d\vv x \le   \int_{[1,w_n]^p}\inddd{ [\vvx] \le K r_n}d\vv x 
\]
and
\[
G_n = \int_{[1,\infty)^p} \inddd{(\floor{x_1}+1) \cdots (\floor{x_p}+1)\le K r_n, \floor{ x_i} +1\le w_n,i=1,\dots,p} d\vv x \ge   \int_{ [2,w_n]^p} \inddd{ [\vvx] \le K r_n} d\vv x.
\]
So for a fixed constant $a>0$, we are left to compute the asymptotics  of
\[
G_n(a):=\int_{ [a,w_n]^p} \inddd{[\vvx] \le K r_n} d\vv x = (w_n-a)^p\proba(U_{n,1}\times\cdots\times U_{n,p} \le Kr_n),
\]
where $U_{n,1},\dots, U_{n,p}$ are i.i.d.~uniform random variables over the interval $(a,w_n)$. The density of products of i.i.d.~uniform random variables can be derived by an induction method \citep{ishihara02distribution} (or, see \citep{dettmann09product} for a simple complex-analysis argument), and one can derive the corresponding cumulative distribution function. In particular, from \citep[(3.9)]{ishihara02distribution} we have, for all $a>0$ fixed,
\begin{align*}
G_n(a) & := (w_n  -a)^p  \proba(U_{n,1}\ldots U_{n,p} \le Kr_n)\\
&  = \summ s0p\binom{p}{s} (-1)^{s}\pp{ K r_n\sum_{j=1}^{p-1}\frac{(-1)^{j+1}}{(p-j)!} \left(\log\frac{Kr_n}{ a^{p-s} w_n^{s}} \right)_+^{p-j} +(-1)^{p-1}(Kr_n- a^{p-s} w_n^s )_+  }\\
& \sim \frac{Kr_n}{(p-1)!}\summ s0p(-1)^s \binom ps\pp{\log\frac{r_n}{w_n^s}}^{p-1}_+ \sim \frac {K }{(p-1)!}\sum_{s=q_{\beta,p}}^p(-1)^{p-s} \binom ps(-\beta_{s})^{p-1} \cdot \frac{w_n^p}n,
\end{align*}
where for the last asymptotic equivalence we have used \eqref{eq:r_n sub}. Note that the last expression does not depend on $a$.
In summary, we have proved $C_{n,1}(K)\sim A_n\sim G_n/p! \sim G_n(1)/p!$ and the desired result follows.
\end{proof}
\begin{Lem}\label{lem:Cn2}
For $r=1,\dots,p-1$, we have 
\equh\label{eq:Cn}
C_{n,2}(r,K) \le \begin{cases}
\displaystyle C \frac{w_n^{2p}}{n^2} \log^{-2r} n, & \mbox{ if } \beta_r>0,\\\\
\displaystyle C \frac{w_n^{2p}}{n^2} \log^{-2r} n, & \mbox{ if } \beta_r = 0,\\\\
\displaystyle C \frac{w_n^{2p-r}}{n}\log ^{r-p}n, & \mbox{ if } \beta_r<0.
\end{cases}
\eque
\end{Lem}
\begin{proof}
Write 
\begin{align}
C_{n,2}(r,K) 
& = \sum_{\substack{\vvi,\vvi'\in\calH(n,K)\\|\vvi\cap\vvi'|   = r}} 1\le  C \sum_{\vvi\in\calH(n,K)}\sum_{\substack{\vvi'\in \calH(n,K) \\
|\vvi\cap\vvi'| = r,i_s'\le i_s, s=1,\dots,p}} 1 \nonumber\\
& \le C \binom pr\sum_{\vvi\in\calH(n,K)}[i_{r+1:p}]\le C\sum_{i_{r+1}<\cdots<i_p\le w_n}[i_{r+1:p}] \abs{\calD_r^*\pp{\frac{r_n}{[i_{r+1:p}]}}}\nonumber\\
& \le C r_n\log^{r-1}r_n\sum_{i_{r+1}<\cdots<i_p\le w_n} 1\le C r_nw_n^{p-r}\log ^{r-1}r_n\sim C \frac{w_n^{2p-r}}n\frac1{\log^{p-r}r_n},\label{eq:Cn2'}
\end{align}
where we have used \eqref{eq:Dp} in the first inequality of the last line.

On the other hand,  using an integral re-expression of sum as in the proof of Lemma \ref{lem:D_p*}, and  the fact $u/2\le \floor{u}$ for $u\ge 1$, we   have  (recall $[\vv v]$ denotes product of all components in $\vv v$)
\begin{align}
C_{n,2}(r,K) &\le \int_{\substack{\vv x,\vv y\in[1,\infty)^{p-r}\\\vv z\in[1,\infty)^r}} \inddd{[\vv x][\vv z]\le 2^pKr_n, [\vv y][\vv z]\le 2^pKr_n}
d\vv xd\vv yd\vv z\nonumber
 \\ &=  \int_{\substack{\vv z\in[1,\infty)^r,[\vv z]\le 2^pKr_n}} \left(\int_{\vv x\in[1,\infty)^{p-r}} \inddd{[\vv x]\le 2^pKr_n/[\vv z]} d\vv x\right)^{2}d\vv z\nonumber\\
 &\le C r_n^2 \log(r_n)^{2(p-r-1)}\int_{\substack{\vv z\in[1,\infty)^r,[\vv z]\le 2^pKr_n}} \frac{1}{[\vv z]^2}  d\vv z \le C r_n^2 \log(r_n)^{2(p-r-1)},\label{eq:Cn2}
\end{align}
where in the third step we applied \eqref{eq:simplex limit}. 

With our choice of $r_n$, $\log r_n\sim C\log n$. Then, 
the bound \eqref{eq:Cn2'} (\eqref{eq:Cn2} resp.) is sharper if $\beta_r<0$ ($\beta_r>0$ resp.).
If $r<p, \beta_r = 0$, then \eqref{eq:Cn2'} yields an upper bound $Cw_n^{2p}\log^{r-p}n/n^2$, and \eqref{eq:Cn2} $Cw_n^{2p}\log ^{-2r}n/n^2$. We simply use the latter upper bound for our purpose later, even it might not always be better than the former. The desired \eqref{eq:Cn} now follows.
\end{proof}
\begin{Lem}\label{lem:rho}
For $\rho_n$ in \eqref{eq:rho_n}, and with any $d_n\to\infty$ as $n\rightarrow\infty$, we have 
\equh\label{eq:c_p}
\rho_n\sim \frac12\mathfrak q_{F,p} \frac n{k_nw_n^p}.
\eque
where $\mathfrak q_{F,p}$ is as in \eqref{eq:c_Fp}.
\end{Lem}
\begin{proof}
Introduce
\begin{align*}
q_{n,1}\topp p  &:= \proba\pp{R_{n,(1,\dots,p)}\cap\{1,\dots, d_n\}\ne\emptyset,\max R_{n,(1,\dots,p)}\le d_n},\\
q_{n,2}\topp p& := \proba\pp{R_{n,(1,\dots,p)}\cap\{1,\dots,d_n\}\ne\emptyset, \max R_{n,(1,\dots,p)}>d_n}.
\end{align*}
Then, 
$q_{n,1}\topp p\le 2\rho_n\le q_{n,1}\topp p+q_{n,2}\topp p$. 
For $q_{n,1}\topp p$ we can apply the last-renewal decomposition and Markov property:
\begin{align*}
q_{n,1}\topp p & = \summ i1{d_n}\proba\pp{\max R_{n,(1,\dots,p)}= i} \\
 & = \summ i1{d_n} \proba\pp{\max R_{n,(1,\dots,p)} = i\mmid i\in R_{n,(1,\dots,p)}}\proba\pp{i\in R_{n,(1,\dots,p)}} =  \summ i1{d_n} \mathfrak q_{F,p} \frac1{w_n^p} = \frac{\mathfrak q_{F,p}d_n}{w_n^p}.
\end{align*}
For $q_{n,2}\topp p$, write first by a similar decomposition based on the last renewal before time $d_n$, 
\begin{align*}
q_{n,2}\topp p & = \summ i1{d_n}\proba\pp{\max(R_{n,(1,\dots,p)}\cap\{1,\dots,d_n\}) = i, \max R_{n,(1,\dots,p)}>d_n} \\
& \le \summ i1{d_n} \proba\pp{i\in R_{n,(1,\dots,p)}, \max R_{n,(1,\dots,p)}>d_n} \\
& = \summ i1{d_n} \proba\pp{i\in R_{n,(1,\dots,p)}}\proba\pp{\max R_{n,(1,\dots,p)}>d_n\mmid i\in R_{n,(1,\dots,p)}} \le \summ i1{d_n} \frac1{w_n^p}\sum_{j=d_n-i+1}^\infty u(j)^p.
\end{align*}
The 
last step above
 follows from the renewal property and 
then the union bound.
With $v(i) := \sum_{j=i}^\infty u(j)^p\downarrow 0$ as $i\to\infty$, the last expression  becomes 
$w_n^{-p}\summ i1{d_n} v(i) = d_nw_n^{-p}\summ i1{d_n}(v(i)/d_n) = d_nw_n^{-p}o(1) = o(q_{n,1}\topp p)$.
We have thus proved \eqref{eq:c_p}.
\end{proof}
\begin{Lem}\label{lem:rho_n(r)}
With $\beta_p<0$, if $d_n\to\infty,n/d_n\to\infty$ as $n\rightarrow\infty$,  then  $\rho_n(r)$ in \eqref{eq:rho_n(r)} satisfies 
\equh\label{eq:rho_n(r) bound} 
\rho_n(r) \le  
\begin{cases}
\displaystyle\frac{Cn^{1+\beta_r}  }{w_n^{2p-r}},  &\mbox{ if } \beta_r>0,
\\\\
\displaystyle \frac {Cn\log n}{w_n^{2p-r}}, & \mbox{ if } \beta_r = 0,\\\\
\displaystyle \frac{C n}{w_n^{2p-r}},  &\mbox{ if } \beta_r<0,r<p,\\\\
\displaystyle \frac{C nd_n^{-(|\beta_p|\vee 1)}\log d_n }{w_n^p}, & \mbox{ if } r = p.
\end{cases}
\eque
\end{Lem}
In fact, the $\log d_n$ term in the case $r=p$ can be dropped if $\beta_p\ne-1$. We keep it for all values of $\beta_p$ for the sake of simplicity.
\begin{proof}
We shall often use  the following fact: $u(n)\le C n^{\beta-1}$ for some constant $C>0$ for all $n\in\N$ (recall \eqref{eq:u(n)}). 
We first show that for all $\vvi,\vvi'$, 
\equh\label{eq:off diagonal}
\sum_{\substack{j,j'=1,\dots,k_n\\|j-j'|>1}}\rho_{n,\vvi,\vvi',j,j'}
\le 
\begin{cases}
\displaystyle C \frac{n^{1+\beta_r}}{w_n^{2p-r}},& \mbox{ if } \beta_r>0,\\\\
\displaystyle C\frac{n\log k_n}{w_n^{2p-r}}, & \mbox{ if } \beta_r = 0,\\\\
\displaystyle C \frac{nd_n^{\beta_r}}{w_n^{2p-r}},& \mbox{ if } \beta_r<0,
\end{cases}
\qmwith r = |\vvi\cap\vvi'|.
\eque
The relation above can be obtained in the following a unified argument for all 3 cases. Write 
\begin{align*}
\rho_{n,\vvi,\vvi',j,j'}&\le\proba\pp{R_{n,\vvi}\cap \calI_{n,j}\neq \emptyset,R_{n,\vvi'}\cap \calI_{n,j'}\neq \emptyset} \le\sum_{k\in \calI_{n,j} }\sum_{k'\in \calI_{n,j'} } \proba\pp{k\in  R_{n,\vvi}, k'\in R_{n,\vvi'}}\\
&= \frac{1}{w_n^{2p-2r}}\sum_{k\in \calI_{n,j} }\sum_{k'\in \calI_{n,j'} } \proba\pp{k,k'\in  R_{n,\vvi\cap \vvi'}}=\frac{1}{w_n^{2p-2r}}\sum_{k\in \calI_{n,j} }\sum_{k'\in \calI_{n,j'} }   \pp{\frac{u(|k-k'|)}{w_n}}^r \\
&\le  \frac{Cd_n^2}{w_n^{2p-2r}}  \pp{\frac{(d_n(|j-j'|-1))^{\beta-1}}{w_n}}^r 
 \le \frac{C d_n^{1+\beta_r}}{w_n^{2p-r}}|j-j'|^{\beta_r-1}, \quad |j-j'|>1,
\end{align*}
and hence
\[
\sum_{\substack{j,j'=1,\dots,k_n\\|j-j'|>1}}\rho_{n,\vvi,\vvi',j,j'}\le C \frac{d_n^{1+\beta_r}}{w_n^{2p-r}}\summ j1{k_n}(k_n-j)j^{\beta_r-1}.
\]
Note   $k_n\sim n/d_n\to\infty$. 
Now, \eqref{eq:off diagonal} follows by applying
\[
\summ j1{n}(n-j)j^{\gamma-1}\le \begin{cases}
 Cn^{\gamma+1}, & \mbox{ if } \gamma >0,\\
C n \log n, &  \mbox{ if } \gamma = 0,\\
 Cn, & \mbox{ if } \gamma<0  .
\end{cases}
\]
For $|j-j
'|\le 1$, consider first $j' = j+1$. Then, with $r = |\vvi\cap\vvi'|$, 
\[
\rho_{n,\vvi,\vvi',j,j+1} \le \sum_{\ell=1}^{2d_n}\sum_{\substack{k\in\calI_{n,j},k'\in\calI_{n,j+1}\\|k-k'|=\ell}}\frac{u(\ell)^r}{w_n^{2p-r}} \le \frac{C}{w_n^{2p-r}}\sum_{\ell=1}^{2d_n}\ell\cdot\ell^{\beta_r-1}\le 
\begin{cases}
\displaystyle\frac {Cd_n^{\beta_r+1}}{w_n^{2p-r}}, & \mbox{ if } \beta_r>-1,\\\\
\displaystyle \frac{C\log d_n}{w_n^{2p-r}}, & \mbox{ if } \beta_r = -1,\\\\
\displaystyle\frac{C}{w_n^{2p-r}}, & \mbox{ if } \beta_r<-1. 
\end{cases}
\]
Similarly,  
\[
\rho_{n,\vvi,\vvi',j,j} \le \sum_{\ell=1}^{d_n}\sum_{\substack{k,k'\in\calI_{n,j}\\|k-k'|=\ell}}\frac{u(\ell)^r}{w_n^{2p-r}} \le \frac{C}{w_n^{2p-r}}\sum_{\ell=1}^{d_n}(d_n-\ell)\ell^{\beta_r-1}\le 
\begin{cases}
\displaystyle\frac {Cd_n^{\beta_r+1}}{w_n^{2p-r}}, & \mbox{ if } \beta_r>0,\\\\
\displaystyle \frac{Cd_n\log d_n}{w_n^{2p-r}}, & \mbox{ if } \beta_r = 0,\\\\
\displaystyle\frac{Cd_n}{w_n^{2p-r}}, & \mbox{ if } \beta_r<0. 
\end{cases}
\]
One readily checks that the bounds for $\rho_{n,\vvi,\vvi',j,j}$ are of equal or larger order than those for $\rho_{n,\vvi,\vvi',j,j+1}$, regardless of the values of $\beta_r$. Therefore, we arrive at
\equh
\sum_{\substack{j,j'=1,\dots,k_n\\|j-j'|\le 1}}\rho_{n,\vvi,\vvi',j,j'}  \le Ck_n\rho_{n,\vvi,\vvi',1,1}\le 
\begin{cases}
\displaystyle\frac {Cnd_n^{\beta_r}}{w_n^{2p-r}}, & \mbox{ if } \beta_r>0,\\\\
\displaystyle \frac{Cn\log d_n}{w_n^{2p-r}}, & \mbox{ if } \beta_r = 0,\\\\
\displaystyle\frac{Cn}{w_n^{2p-r}}, & \mbox{ if } \beta_r<0,r<p.
\end{cases}
\label{eq:diagonal new}
\eque
For the case $r=p$, 
recall that the summation now excludes $j=j'$ and hence in addition to \eqref{eq:off diagonal} we only consider $|j-j'| = 1$. Then, by the bound on $\rho_{n,\vvi,\vvi',j,j+1}$ above, we have
\equh\label{eq:diagonal new'}
\sum_{\substack{j,j'=1,\dots,k_n\\|j-j'|= 1}}\rho_{n,\vvi,\vvi,j,j'} \le \frac {Cnd_n^{-(|\beta_p|\vee 1)}}{w_n^{p}}\log d_n,
\eque
and the $\log d_n$ factor can be dropped if $\beta_p\ne -1$. 
Combining \eqref{eq:off diagonal}, \eqref{eq:diagonal new} and \eqref{eq:diagonal new'}, we obtain \eqref{eq:rho_n(r) bound}. 
\end{proof}
\begin{proof}[Proof of Lemma \ref{lem:2}]
Now we complete the proof of \eqref{eq:goal two moments}. It is straightforward to show
\[
\limn\esp\wb\eta_{n,K}(E_s) = \esp \xi_K(E_s),\ s=1,\dots,d,
\]
Indeed, consider $E = (y_0,y_1]\times (a,b]$ or $(y_0,y_1]\times [0,b]$ (with $a=0$ in the latter case). Then from \eqref{eq:esp eta bar}, by  Lemma  \ref{lem:Cn} (with $y_i^{-1/\alpha}$, $i=0,1$, playing the role of $K$ there),  Lemma \ref{lem:rho} and the restriction $K^{-1/\alpha}<y_0<y_1$, we have
\begin{align*}
\esp\wb\eta_{n,K}(E) & \sim  \pp{C_{n,1}(y_0^{-\alpha})-C_{n,1}(y_1^{-\alpha})}\cdot (b-a)k_n\cdot \rho_n \\
& \sim (y_0^{-\alpha}-y_1^{-\alpha})\frac{\mathsf D_{\beta,p}}{p!(p-1)!}\frac{w_n^p}n\cdot (b-a)k_n \cdot\frac12\mathfrak q_{F,p} \frac n{k_nw_n^p}\\
&\to \frac1{2p!(p-1)!}(y_0^{-\alpha}-y_1^{-\alpha})\mathsf D_{\beta,p}\mathfrak q_{F,p} (b-a)= \esp \xi_K(E).
\end{align*}
Next, among the bounds for $C_{n,2}(r,K)$ in Lemma \ref{lem:Cn2},  the   $r=1$ (note that $\beta_1=\beta>0$) case is of the dominating order  for $r=1,\ldots,p-1$. Recall also that $C_{n,2}(p,K) = C_{n,1}(K)$, for which we have the precise estimate in Lemma \ref{lem:Cn}. Combining these into \eqref{eq:b_n1}  we have
\begin{align*}
b_{n,1}(K) & \le k_n^2 \pp{C C_{n,2}(1,K)+C_{n,1}(K)}\rho_n^2\\
& \le Ck_n^2 \pp{\frac{w_n^{2p}}{n^2}\log^{-2}n + r_n \log^{p-1}n} \pp{\frac n{k_nw_n^p}}^2 \le C\left(\frac {1}{\log^{2}n} + n^{\beta_p} \right)\rightarrow 0
\end{align*}
as $n\rightarrow\infty$,
where for the last inequality we have used \eqref{eq:r_n sub}.
At last by Lemmas \ref{lem:Cn}, \ref{lem:Cn2}, \ref{lem:rho_n(r)}, and the assumption $w_n \sim \frac{\mathsf C_F}{1-\beta} n^{1-\beta}$,   
\[ 
\rho_n(r) C_n(r,K) \le  
\begin{cases}
\displaystyle   C\log^{-2r}(n),  &\mbox{ if } \beta_r>0,
\\\\
\displaystyle   C\log^{1-2r}(n), & \mbox{ if } \beta_r = 0,\\\\
\displaystyle C\log ^{r-p}n  ,  &\mbox{ if } \beta_r<0,r<p,\\\\
\displaystyle   C  d_n^{-(|\beta_p|\vee 1)}   \log d_n, & \mbox{ if } r = p,
\end{cases}
\]
  and  in view of \eqref{eq:b_n2}, we have $\limn b_{n,2}(K) = 0$.
 We have thereby completed the proof of  \eqref{eq:goal two moments} and  hence \eqref{eq:Poisson fdd}.

The fact that \eqref{eq:Poisson fdd}   implies the weak convergence of the point processes $\wb\eta_{n,K}\weakto \xi_K$ in $\mathfrak M_p((0,\infty]\times[0,1])$  as $n\rightarrow\infty$ is well-known (e.g., by combining \citet[Theorem 4.11(ii)]{kallenberg17random} with Cram\'er--Wold Theorem.) Note that we shall need to apply \eqref{eq:Poisson fdd}  with a relaxed assumption on the test sets $E = (y_0,y_1]\times \wt{I}$ so that $y_0>0$ (instead of $y_0>K^{-1/\alpha}$ above). The modification does not effect the proof above as for the point processes of interest, there are no points of which the first coordinate has a value less than $K^{-1/\alpha}$ by our construction.
\end{proof}
\subsection{Proof of Proposition \ref{prop:1}}\label{sec:approximation}
Recall the overview of the approximations we shall work with in \eqref{eq:structure}.  

Next, we approximate $\xi_{n,K}$ by $\wb\eta_{n,K}$, with an intermediate approximation by $\eta_{n,K}$ as seen  in the proof. Recall the set $E$ in \eqref{eq:E}. Introduce accordingly
\[
E^{\pm,\epsilon}:=(y_0(1\mp\epsilon)^{1/\alpha},y_1(1\pm\epsilon)^{1/\alpha}]\times \wt{I},
\]
where $\wt{I}=(a,b]$ or $[0,b]$ (with $a=0$ in the latter case), $0\le a<b\le 1$.
We shall assume $0<y_0<y_1$ and eventually  $\epsilon>0$ small enough so that $y_0(1+\epsilon)^{1/\alpha}<y_1(1-\epsilon)^{1/\alpha}$. Otherwise, by convention $E^{-,\epsilon} = \emptyset$.
\begin{Lem}\label{lem:1}
With the notations above, for all $\epsilon>0$, with probability going to 1 as $n\to\infty$,
\[
\wb \eta_{n,K}(E^{-,\epsilon})\le \xi_{n,K}(E)\le \wb\eta_{n,K}(E^{+,\epsilon}).
\]
\end{Lem}
\begin{proof}
Consider the following events
 \begin{align*} 
\Omega_{n,1}(K) & :=\bigcap_{j=1}^{k_n}\ccbb{\exists \mbox{ at most one $\vvi\in\calH(n,K)$ s.t.~$R_{n,\vvi}\cap \calI_{n,j}\ne\emptyset$}},\\
 \Omega_{n,2}(K,\epsilon) &:=\ccbb{\frac{[\Gamma_\vvi]}{[\vvi]}\in(1-\epsilon,1+\epsilon), \mfa \vvi\in \calH(n,K), R_{n,\vvi}\ne\emptyset}.
\end{align*}

Recall $
\what m_{n,j}\topp K/c_n =\max_{k\in\calI_{n,j}}\sum_{\vvi\in\calH(n,K)} 
\frac{[\varepsilon_\vvi]}{([\Gamma_\vvi]/r_n)^{1/\alpha}}\inddd{k\in R_{n,\vvi}}$.
Fix  $K>0$.  Then, we claim that under  $\Omega_{n,1}(K)$,   for each $j=1,\ldots,k_n$, the following two subsets of $(0,\infty]$ coincide:
\[\{\what m_{n,j}\topp K/c_n:  \what m_{n,j}\topp K  > 0 \} =\{([\Gamma_\vvi]/r_n)^{-1/\alpha}:   \vvi\in\calH(n,K),\ R_{n,\vvi}\cap \calI_{n,j}\ne\emptyset,\ [\varepsilon_\vvi]=1\}.
\] 
In particular, both are empty if   for any $\vvi \in \calH(n,K)$, we have either $R_{n,\vvi}\cap \calI_{n,j}=\emptyset$  or $[\varepsilon_\vvi]=-1$; otherwise, both are the same singleton set $\{([\Gamma_\vvi]/r_n)^{-1/\alpha}\}$ where $\vvi$  is the  unique index under $\Omega_{n,1}(K)$  for which $R_{n,\vvi}\cap \calI_{n,j}\ne\emptyset$. It follows that  $\xi_{n,K} = \eta_{n,K}$ under $\Omega_{n,1}(K)$.

Next,  note that the only difference between $\wb\eta_{n,K}$ and $\eta_{n,K}$ is that for each $\vvi$, the random $[\Gamma_\vvi]$ in $\eta_{n,K}$  is replaced by the nonrandom $[\vvi]$ in $\wb\eta_{n,K}$.   If  restricted to $\Omega_{n,1}(K,\epsilon)\cap \Omega_{n,2}(K,\epsilon)$, it easily follows from monotonicity  that $\wb \eta_{n,K}(E^{-,\epsilon})\le \eta_{n,K}(E)\le \wb\eta_{n,K}(E^{+,\epsilon})$. 

Combining the above, to prove the desired result, it remains to show that the   events $\Omega_{n,1}(K)$  and $\Omega_{n,2}(K)$ each occur with probability tending to one as $n\to\infty$. 

We first prove $\limn \proba\pp{\Omega_{n,1}(K)} =1$. This 
follows from the fact that the limit of $\wb\eta_{n,K}$   is Poisson. 
 A more detailed argument  can be given by using the estimates needed in the Poisson approximation. Indeed, setting $N_n(K)= \summ j 1{k_n} \sum_{\vvi \ne\vvi'\in\calH(n,K)}\inddd{R_{n,\vvi}\cap \calI_{n,j}\ne\emptyset,R_{n,\vvi'}\cap\calI_{n,j}\ne\emptyset}$, then 
\begin{align}
\proba\pp{\Omega_{n,1}(K)^c}  & = \proba \pp{N_n(K)\ge 1}\le \esp N_n(K)=  \summ j1{k_n} \sum_{\vvi \ne\vvi'\in\calH(n,K)} \rho_{n,\vvi,\vvi',j,j} \notag \\ & \le  
C_{n,1}(K)^2 k_n \rho_n^2+\summ r1{p-1}C_{n,2}(r,K)\rho_n(r) \notag
\end{align}
where the last bound is obtained by dividing the double sum over $(\vvi,\vvi')$ into two cases: $\vvi\cap \vvi'=\emptyset$ and $\vvi\cap \vvi'\neq \emptyset$. The first term in the bound, in view of Lemmas \ref{lem:Cn} and \ref{lem:rho},  is bounded up to a constant by $k_n^{-1}\rightarrow 0$ as $n\rightarrow \infty$. The second term in the bound is $b_{n,2}(K)$ in \eqref{eq:b_n2} which tends to zero as $n\rightarrow\infty$ as already shown.

At last, we show that $\limn\proba(\Omega_{n,2}(K,\epsilon)) = 1$. As an intermediate step, we claim that for all $m\in\N$ fixed, 
\[
\limn\proba\pp{\exists  \vvi = (i_1,\dots,i_p)\in\calH(n,K), \mbox{ s.t. } R_{n,\vvi}\ne\emptyset, i_1\le m } = 0.
\]
Indeed, by a union bound, Lemma \ref{lem:rho},  \eqref{eq:r_n sub} and  \eqref{eq:Cn1} ($p$ there replaced by $p-1$), the probability displayed above is bounded by  
\begin{align}
n w_n^{-p}|\{ \vvi\in\calH(n,K): i_1\le m\}|
 &\le n w_n^{-p} m    \abs{\ccbb{(i_2,\ldots,i_p)\in \calD^*_{p-1}(Kr_n ):  \ i_p\le w_n}}
\notag \\ &\le    \frac{C}{\log(n)}\rightarrow 0  \label{eq:Omegan2}
\end{align}
as $n\rightarrow\infty$.
 Then, the desired result follows from the strong law of large numbers: $\lim_{m\rightarrow\infty}\sup_{i> m} |\Gamma_i/i - 1|=0$ almost surely.
\end{proof}
\begin{proof}[Proof of Proposition \ref{prop:1}]
Consider any collection of $E_1,\dots,E_d$ as in Lemma \ref{lem:2} above. Then, based on  Lemmas \ref{lem:1} and \ref{lem:2}, the almost sure convergence $\xi_K(E^{\pm,\epsilon}_s)\rightarrow \xi_K(E_s)$   as $\epsilon\downarrow0$, 
we conclude that  as $n\rightarrow\infty$,
\[
\pp{\xi_{n,K}(E_1),\dots,\xi_{n,K}(E_d)}\weakto \pp{\xi_K(E_1),\dots,\xi_K(E_d)}.
\]
The desired result now follows.
\end{proof}
\subsection{Proof of \eqref{eq:thm sub-crit}}
\label{sec:remainder}
We first provide a uniform control on the remainder of the truncation. 
\begin{Lem}\label{lem:remainder}
We have
\equh\label{eq:uniform_control'}
\lim_{K\to\infty}\limsupn\proba\pp{\frac1{c_n}\max_{k=1,\dots,n}\abs{X_{n,k}-X_{n,k}\topp K}>\epsilon} = 0 \mfa \epsilon>0.
\eque
\end{Lem}
\begin{proof}
Recall $X_{n,k}\topp K = 
\sum_{\vvi\in\calH(n,K)}w_n^{p/\alpha}
\frac{[\varepsilon_\vvi]}{[\Gamma_\vvi]^{1/\alpha}}\inddd{k\in R_{n,\vvi}}$.
With probability tending to one as $n\rightarrow\infty$, we claim that
\[
X_{n,k}\topp K = \wt X_{n,k}\topp K:=\sum_{\vvi\in\calD_p^*(Kr_n)}w_n^{p/\alpha}
\frac{[\varepsilon_\vvi]}{[\Gamma_\vvi]^{1/\alpha}}\inddd{k\in R_{n,\vvi}}, \ k=1,\dots,n.
\]
Indeed,
 Note that if $r_n\le Cw_n$ for some constant $C>0$ (equivalently $\beta_{p-1}\ge 0$), then for $n$ large enough, $\calH(n,K) = \calD_p^*(Kr_n)$, and there is nothing to prove. So assume $\beta_{p-1}<0$.  The key observation is that for $X_{n,k}\topp K$ and $\wt X_{n,k}\topp K$ to differ for some $k$, then necessarily there exists $\vvi=(i_1,\dots,i_{p})\in\calD_p^*(Kr_n)$ such that $i_p>w_n$ and $k\in R_{n,\vvi}$. However, $i_p>w_n$ implies necessarily that $(i_1,\dots,i_{p-1})\in \calD_{p-1}^*(Kr_n/w_n)$. But, 
\equh\label{eq:p-1}
\what \calD_{p-1}^*(Kr_n/w_n):=\ccbb{\vvi\in\calD_{p-1}^*(Kr_n/w_n): R_{n,\vvi}\ne\emptyset}
\eque
is an empty set with probability tending to one as $n\to\infty$, in the case $\beta_{p-1}<0$. Indeed, using Lemma \ref{lem:rho} (with $d_n=n$), \eqref{eq:Dp} and \eqref{eq:r_n sub}, we have
\[
\esp\abs{\what \calD_{p-1}^*(Kr_n/w_n)} \le C \abs{\calD_{p-1}^*(Kr_n/w_n)} \frac n{w_n^{p-1}} \le C \frac{nr_n}{w_n^p}\log^{p-2} r_n\le \frac C{\log n}\to 0
\]
as $n\to\infty$. 

Therefore, it suffices to show \eqref{eq:uniform_control'} with $X_{n,k}\topp K $ replaced by $\wt X_{n,k}$. To do so, we introduce fixed $m\in \N$ to be chosen later. 
   By \eqref{eq:r_n sub}, a triangular inequality, and the   inequality $\max_k (\sum_\ell a_{k\ell})\le \sum_\ell \max_k (a_{k\ell})$, $a_{k\ell}\in \R$,  we have
\begin{align*}
\frac1{c_n}\max_{k=1,\dots,n}\abs{X_{n,k}-\wt X_{n,k}\topp K}&= \max_{k=1,\dots,n} \left| \sum_{\vvi\in \calD_p \setminus\calD_p^*(Kr_n)}r_n^{1/\alpha}
\frac{[\varepsilon_\vvi]}{[\Gamma_\vvi]^{1/\alpha}}\inddd{k\in R_{n,\vvi}}\right|\\ &\le  \summ q0{p-1}\sum_{1\le i_1<\cdots<i_q\le m}\frac1{[\Gamma_{\vvi_{1:q}}]^{1/\alpha}}Z_{n,i_1,\dots,i_q}(K)
\end{align*}
with
\[
Z_{n,i_1,\dots,i_q}(K):= \max_{k=1,\dots,n}  r_n^{1/\alpha}\left|\sum_{m<i_{q+1}<\cdots<i_p, i_1\ldots i_p>Kr_n} \frac{[\varepsilon_{\vvi_{q+1:p}}]}{  [\Gamma_{\vvi_{q+1:p}}]^{1/\alpha}} \inddd{k\in R_{n,\vvi}}\right|,\  q=0,\dots,p-1.
\]
When $q=0$, we understand $Z_{n,i_1,\dots,i_q}(K)$ above  as
\[
 Z_{n,\emptyset}(K):=   \max_{k=1,\dots,n}r_n^{1/\alpha}\abs{ \sum_{m <i_1<\cdots<i_p, i_1\ldots i_p>Kr_n}\frac{[\varepsilon_\vvi]}{[\Gamma_\vvi]^{1/\alpha}}\inddd{k\in R_{n,\vvi}}}.
\] 
Therefore, it suffices to show that for each $q=0,\dots,p-1$ and $i_1<\dots<i_q\le m$ fixed, 
\equh\label{eq:zeta}
\lim_{K\to\infty}\limsupn \proba\pp{Z_{n,i_1,\dots,i_q}(K)>\epsilon} = 0 \mfa \epsilon>0.
\eque
Bounding the max by a sum, using the orthogonality induced by $[\varepsilon_{\vvi_{q+1:p}}]$, the fact $P(k\in R_{n,\vvi} )=w_n^{-p}$ and  \eqref{eq:esp Gamma product bound} (for which it suffices to choose $m\ge 2 p/\alpha$), we have 
\begin{align*}
\esp Z_{n,i_1,\dots,i_q}(K)^2   & = r_n^{2/\alpha} \esp\left( 
\max_{k=1,\dots,n}    \left| \sum_{m<i_{q+1}<\cdots<i_p, [\vvi]>Kr_n} \frac{[\varepsilon_{\vvi_{q+1:p}}]}{  [\Gamma_{\vvi_{q+1:p}}]^{1/\alpha}} \inddd{k\in R_{n,\vvi}}\right|^2\right)   \\ 
&\le  C \frac{nr_n^{2/\alpha}}{w_n^p}\sum_{\substack{m<i_{q+1}<\cdots<i_p\\
 i_{q+1}\cdots i_p>Kr_n/m^{q}}}  (i_{q+1}\cdots i_p)^{-2/\alpha}.
\end{align*}
Now we estimate the multiple sum above by an integral approximation and suitable change of variables:
\begin{align*}
  \sum_{\substack{m<i_{q+1}<\cdots<i_p\\
 i_{q+1}\cdots i_p>Kr_n/m^{q}}}  &(i_{q+1}\cdots i_p)^{-2/\alpha} \le \int_{[m ,\infty)^{p-q}} (x_{q+1}\cdots x_{p})^{-2/\alpha} \inddd{(x_{q+1}+1)\cdots (x_{p}+1)>  Kr_n/m^{q}} d\vv x
\\&=   m^{(p-q)(1-2/\alpha)} \int_{[0,\infty)^{p-q}} e^{(1-2/\alpha)(y_1+\cdots+y_{p-q})} \inddd{(me^{y_1}+1)\cdots(me^{y_{p-q}}+1)>  Kr_n/m^{q}}  d\vv y
\\ & \le m^{(p-q)(1-2/\alpha)} \int_{[0,\infty)^{p-q}} e^{(1-2/\alpha)(y_1+\cdots+y_{p-q})} \inddd{y_1+\cdots+y_{p-q}>\log( K r_n)-p\log(m+1)}  d\vv y \\
&\le  C \int_{\log( K r_n) }^{\infty} e^{(1-2/\alpha) z}  z^{p-q-1} dz
\le C (Kr_n)^{1-2/\alpha}  \log^{p-q-1}(K r_n),
\end{align*}
where  in the second inequality above we   used $m e^{y_j}+1\le (m+1) e^{y_j}$, $j=1,\ldots,p-q$.
Hence with \eqref{eq:r_n sub} one concludes
\begin{align}
\esp Z_{n,i_1,\dots,i_q}(K)^2 & \le C nw_n^{-p}r_n^{2/\alpha} (Kr_n)^{1-2/\alpha} \log^{p-q-1}(K r_n)  \nonumber\\
&\le  \frac{C K^{1-2/\alpha}  \log^{p-q-1}(Kn) }{\log ^{p-1} n},\ q=0,\dots,p-1.\label{eq:1}
\end{align}
So \eqref{eq:zeta} follows   from Markov inequality, and hence   the whole proof is finished.
\end{proof}

\begin{proof}[Proof of \eqref{eq:thm sub-crit}]
Since the limit   $\calM_{\alpha}^{\rm is}(I)>0$ almost surely for any   $I\in\calG_0$, it suffices to show for any $x_i>0$ and any    $I_i\in \calG_0$, $i=1,\ldots,d$, that
\begin{align}
\limn\proba\pp{\frac1{c_n}M_n(I_i)\le x_i, i=1,\dots,d} & = \proba\pp{\calM_\alpha^{\rm is}(I_i)\le x_i, i=1,\dots,d}\nonumber\\
& = \proba\pp{\xi\pp{\bigcup_{i=1}^d\pp{(x_i,\infty]\times I_i}} = 0}.\label{eq:fdd goal subcrit}
\end{align}
We have established  Proposition \ref{prop:1}, which  implies that  
\begin{align}
&\limn \proba\pp{\xi_{n,K}\pp{\bigcup_{i=1}^d\pp{(x_i,\infty]\times I_i}} = 0}=\limn \proba\pp{\frac1{c_n}\max_{j/k_n\in G_i, \what  m_{n,j}\topp K>0}\what m_{n,j}\topp K\le x_i, i=1,\dots,d}\nonumber  \\  =& \limn \proba\pp{\frac1{c_n}\max_{j/k_n\in I_i }\what m_{n,j}\topp K\le x_i, i=1,\dots,d}= \proba\pp{\xi_K\pp{\bigcup_{i=1}^d\pp{(x_i,\infty]\times I_i}} = 0}.\label{eq:wt xi''}
\end{align}
where in the first equality in the line of \eqref{eq:wt xi''}, we have used the assumption $x_i>0$, $i=1,\ldots,d$ and the convention $\max_\emptyset=-\infty$.
This will be shown to imply that
\equh\label{eq:xi_K limit}
\limn \proba\pp{\frac1{c_n}\max_{k/n\in G_i}\what X_{n,k}\topp K\le x_i, i=1,\dots,d} = \proba\pp{\xi_K\pp{\bigcup_{i=1}^d\pp{(x_i,\infty]\times I_i}} = 0}.
\eque
  Indeed, for any open sub-interval $I\subset [0,1]$ and large enough $n$, the two terms  $c_n^{-1}\max_{k/n\in I}X_{n,k}\topp K$ and $ c_n^{-1} \max_{j/k_n\in I}\what m_{n,j}\topp K$ could differ due to two possibilities: first, the sets $\{k\in \N: k/n\in I\}$ and $\{k\in \calI_{n,j}: j/k_n\in I \}$ possibly differ near the boundaries of $I$  with at most $O(b_n)$ different indices; second, if  $\sup I=1$,   the first term involves also the  last partial block $\calI_{n,k_n+1}:=\{d_nk_n+1,\dots,n\}$. 
   Suppose first $I=(a,b)$ with $0<a<b<1$. It is clear that for $\epsilon>0$ small enough so that $0<a-\epsilon<a+\epsilon<b-\epsilon<b+\epsilon<1$, and $n$ large enough, we have
\[
\max_{j/k_n\in (a+\epsilon,b-\epsilon)}\what{m}_{n,j}\topp K\le \max_{k/n\in (a,b)}X_{n,k}\topp K\le \max_{j/k_n\in (a-\epsilon,b+\epsilon)}\what{m}_{n,j}\topp K.
\]
Therefore, \eqref{eq:xi_K limit} follows from a sandwich approximation argument based on \eqref{eq:wt xi''}. Similar arguments applies if $0=\inf I<\sup I<1$. When $\sup I=1$, it suffices to additionally apply  the following fact  due to stationarity: $ c_n^{-1}\max_{k\in \calI_{n,k_n+1}}X_{n,k}\topp K \eqd   c_n^{-1}\max_{k=1,\ldots,n-k_nd_n}X_{n,k}\topp K=o_p(1)$, where the $o_p(1)$ claim again follows from  a similar approximation argument using \eqref{eq:wt xi''}.

Now   \eqref{eq:xi_K limit}  has been verified.    The desired convergence in \eqref{eq:fdd goal subcrit} follows from a standard triangular-array approximation (e.g.~\citep[Theorem 3.2]{billingsley99convergence}) based on  Lemma \ref{lem:remainder}. To do so, notice that when $n$ is large enough so that each of the $\max_{k/n\in I_i}$ below is over an nonempty set,  we   have \[
\abs{\frac1{c_n}\max_{k/n\in I_i}X_{n,k} - \frac1{c_n}\max_{k/n\in I_i}X_{n,k}\topp K}\le \frac1{c_n} \max_{k=1,\dots,n}\abs{X_{n,k}-X_{n,k}\topp K}, \quad i=1,\ldots,d.
\]
In addition, one readily checks $\xi_K\weakto \xi$ so that the last expression of \eqref{eq:wt xi''} tends to the last expression of \eqref{eq:fdd goal subcrit} as $K\to\infty$. Hence the triangular-array approximation is justified.
\end{proof}
\begin{Rem}\label{rem:region}
In Lemma \ref{lem:2},  if we do not work with $\calH(n,K)$ but always with $\calD_p^*(Kr_n)$ when $\beta_{p-1}<0$, then the two-moment method fails. In particular,
the convergence $b_{n,2}(K)\to 0$ (now $b_{n,2}(K)$ is summing over a larger region) no longer holds, and moreover 
 we do not have the same asymptotics for $\esp \wb \eta_{n,K}(E)$. The delicate picture is due to the emergence of conditional clusters of unbounded size. Consider  
 \[
\what \calD_p^*(Kr_n):=  \ccbb{\vvi\in\what \calD_p^*(Kr_n): R_{n,\vvi}\ne\emptyset} = \what \calD_p^{*,1}(Kr_n)\cup\what \calD_p^{*,2}(Kr_n),
 \]
 with
 \begin{align*}
 \what \calD_p^{*,1}(Kr_n) & :  =  \ccbb{\vvi\in\what \calD_p^*(Kr_n): R_{n,\vvi}\ne\emptyset, i_p\le w_n},\\
 \what \calD_p^{*,2}(Kr_n) & :  =  \ccbb{\vvi\in\what \calD_p^*(Kr_n): R_{n,\vvi}\ne\emptyset, i_p> w_n},
 \end{align*}
and
recall $\what \calD_{p-1}^*(Kr_n/w_n)$ as in \eqref{eq:p-1}. We have shown that $\what \calD_{p-1}^*(Kr_n/w_n)$ is an empty set as $n\to\infty$ with probability going to one, and hence so is the set $\what \calD_p^{*,2}(Kr_n)$ (so there is no contribution from random variables associated to $\what \calD_p^{*,2}(Kr_n)$ in the Poisson limit). However, one can verify that $\esp|\what \calD_p^{*,i}(Kr_n)|, i=1,2$, are of the same order $ w_n^p/n $. This necessarily implies that given that $\what \calD_p^{*,2}(Kr_n)$ is non-empty (say with probability $p_n\downarrow 0$),   its size is of order $1/p_n\to\infty$. 

Similar phenomena are known in limit theorems for extremes for other models with delicate dependence. See for example \citep{kistler15derrida} for the explanation of a similar situation regarding extremes for the random energy model. 
\end{Rem}

\section{Convergence for random sup-measures,  $\beta_p =0$}\label{sec:critical}

Throughout this section, assume $\beta_p=0$, or equivalently
\[
\beta = \frac{p-1}p, \quad p\ge2.
\]
The goal of this section is to prove the following case of Theorem \ref{thm:RSM} (recall the  characterization of the weak convergence in the space of sup-measures   discussed around \eqref{eq:fdd RSM conv}): with $\beta_p = 0$:
\equh\label{eq:thm crit}
\frac1{c_n}\ccbb{M_n(I)}_{I\in\calG_0}\fddto
\ccbb{\mathfrak C_{F,p}^{1/\alpha}\calM_{\alpha}^{\rm is}(I)}_{I\in\calG_0}
\eque
where
\equh\label{eq:cFp crit}
c_n = \pp{\frac{n(\log\log n)^{p-1}}{\log n}}^{1/\alpha} \qmand 
\mathfrak C_{F,p} = \frac12\frac{(\mathsf C_F\Gamma(\beta)\Gamma(1-\beta))^p}{p!(p-1)!}.
\eque
Again we start by presenting an overview of the approximations.
Recall that $\calI_{n,j} = \{(j-1)d_n+1,\dots,jd_n\}, j=1,\dots,k_n = \floor{n/d_n}$, but this time we shall specifically work with
\equh\label{eq:d_n critical}
d_n \sim \frac n{\log ^ p n} \qmand k_n \sim \frac n{d_n} \sim \log ^{p}n.
\eque
Accordingly, we have now
\equh\label{eq:r_n critical}
r_n = \frac{w_n^p}{c_n^\alpha}\sim  \pp{\frac{\mathsf C_F}{1-\beta}}^p\frac{ \log n}{(\log\log n)^{p-1}}.
\eque
We emphasize that we use the same symbols $d_n, k_n, c_n, r_n$ as in Section \ref{sec:sub-critical}, but they are defined differently here. We also use the same definitions for random variables $\xi_{n,K}, \eta_{n,K}, \wb\eta_{n,K}$ defined as in the previous section, but now with the rates chosen above. 

The key behind \eqref{eq:thm crit} is again the following (in place of Proposition \ref{prop:1}), but now with a different normalizing constant in the limit $\xi$ and hence $\xi_K$.
\begin{Prop}\label{prop:2}
With $\beta_p<0$, $d_n\to\infty$,
\[
 \xi_{n,K}  := \summ j1{k_n}\ddelta{\what m_{n,j}\topp K/c_n,j/k_n}\inddd{\what m_{n,j}\topp K> 0} \weakto  \xi_K:=\sif\ell1 \ddelta{\mathfrak C_{F,p}
 ^{1/\alpha}
  \Gamma_\ell^{-1/\alpha},U_\ell}\inddd{\mathfrak C_{F,p} 
 ^{1/\alpha}
 \Gamma_\ell^{-1/\alpha}>K},
\]
in $\mathfrak M_p((0,\infty]\times [0,1])$, where $\mathfrak C_{F,p}$ is as in \eqref{eq:cFp crit}. 
\end{Prop}
The proof is similar as the one with $\beta_p<0$ in Section \ref{sec:sub-critical}, and  proceeds by a series of approximations as summarized below in place of  \eqref{eq:structure} and \eqref{eq:structure1}. For the proof of Proposition \ref{prop:2}, for some test set $E$ as in \eqref{eq:E}, we have
\[
\xi_{n,K}(E) \quad\underbrace{\approx\quad \eta_{n,K}(E)\quad \approx}_{\rm Lemma~\ref{lem:1 crit}}\quad \wb\eta_{n,K}(E) \underbrace{\weakto}_{\rm Lemma~\ref{lem:2 crit}}\xi_K(E).
\]
Then, for the proof of, say the marginal convergence in \eqref{eq:thm crit}, this time we have
\[
\proba\pp{M_n(G)\le y} \underbrace\approx_{\rm Lemma~\ref{lem:remainder1}} \proba\pp{M_{n,K}(G)\le y} \approx
\proba\pp{ \xi_{n,K}( (y,\infty]\times G) = 0}.
\]
 Lemmas \ref{lem:1 crit}, 
    \ref{lem:2 crit} and \ref{lem:remainder1} respectively correspond to   Lemmas \ref{lem:1}, \ref{lem:2} and  \ref{lem:remainder}  for the subcrtical case $\beta_p<0$ in Section \ref{sec:sub-critical}.

\begin{Rem}\label{rem:difference}
 Despite the fact that in both regimes the limit extremes are independently scattered, 
geometries underlying the Poisson approximation for local extremes are different. Therefore so are the proofs for Lemmas   \ref{lem:2 crit} and \ref{lem:remainder1}  and Lemmas \ref{lem:2} and  \ref{lem:remainder} mentioned above. The different geometries can be viewed from the following two aspects.
  \begin{enumerate}[(a)]
\item The shape of the region that contributes in each regime is different. In the sub-critical regime, it is possible that $r_n\gg w_n$ (this is the case if $\beta_{p-1}<0$, and the shape has a delicate influence in the limit through the constant $\mathsf D_{\beta,p}$), while at critical regime $r_n\ll w_n$ by our assumption. In particular, for $n$ large enough $\calH(n,K) = \calD_p^*(Kr_n)$ and there is no need to work with $\calH(n,K)$. 
\item 
Second, in terms of the size $d_n$ reflecting the size of each local extremal cluster, we see that the order is drastically different from the sub-critical regime. Recall that in that regime only $d_n\to\infty$ and $n/d_n\to\infty$ are assumed. Here, however, we have more restrictive constraints on both upper and lower bounds of the rate. Indeed, in order to apply the two-moment method we actually need (see Lemmas \ref{lem:Cn critical} and \ref{lem:rho_n(r) critical} below)
\equh\label{eq:k_n}
k_n^{\beta_{p-1}} = k_n^{1/p}\ge C\log n \qmand \log k_n = o(\log d_n). 
\eque
For a lower bound on the rate of $d_n\to\infty$, 
the first part of Lemma \ref{lem:rho_n(r) critical} enforces the constraint
\[
C_{n,1}(K)k_n\proba(R_{n,\vvi}\cap\calI_{n,1}) \sim C\frac{\log n \cdot k_nd_n}{n\log d_n} = O(1),
\]
which implies already a lower bound $\log d_n\ge C \log n$ on the rate $d_n\to\infty$. The estimate above is sharp and hence $d_n\to\infty$ {\em has to grow at least  at a polynomial rate}, which is different from the sub-critical regime.
\end{enumerate}
\end{Rem}
\begin{Rem}
 The more restrictive second constraint in \eqref{eq:k_n} on the lower bound, coming from \eqref{eq:critical off diagonal r=p},  requires that $d_n\sim nL(n)$ for some slowly varying function $L$ as $n\to\infty$.
We do not know whether \eqref{eq:critical off diagonal r=p} is optimal.
For the upper bound on $d_n\to\infty$, the first constraint in \eqref{eq:k_n}, coming from Lemma \ref{lem:critical same block} below, is more restrictive than $n/d_n\to\infty$. We do not know if our estimate on this part is optimal either. 
\end{Rem}

\subsection{Two-moment method}
\begin{Lem}\label{lem:2 crit}
With $\beta_p=0$, we have
\[
 \wb\eta_{n,K}
= \summ j1{k_n}\sum_{\vvi\in \calD^*_p(n,K)}\ddelta{([\vv i]/ r_n)^{-1/\alpha},j/k_n} \inddd{R_{n,\vvi}\cap \calI_{n,j}\ne\emptyset,\ [\varepsilon_\vvi] =1}
\weakto\xi_K
\]
as $n\to\infty$ in $\mathfrak M_p((0,\infty]\times[0,1])$. 
\end{Lem}
Again we check the conditions of two-moments method as  in \eqref{eq:goal two moments}. Recall that $C_{n,1}(K)$, $C_{n,2}(r,K)$, $\rho_n$ and $\rho_n(r)$ are defined the same way as before (see \eqref{eq:Cn1 def}, \eqref{eq:Cn2 def}, \eqref{eq:rho_n}, \eqref{eq:rho_n i j}), but with different rates $r_n, c_n, d_n, k_n$. 
The following estimates collect results that are similar as in Section \ref{sec:sub-critical}.
\begin{Lem}\label{lem:Cn critical}
We have
\[
C_{n,1}(K)\sim \frac{K }{p!(p-1)!}\pp{\frac{\mathsf C_F}{1-\beta}}^p\log n, 
\]
and
\equh\label{eq:Cn2 critical}
C_{n,2}(r,K) \le C\frac{\log^{2}n}{(\log\log n)^{2r}},r=1,\dots,p-1.
\eque
Furthermore, as $n$ and $d_n \rightarrow \infty$,  
\equh\label{eq:rho_n crit}
\rho_n=\frac12\cdot \proba\pp{R_{n,\vvi}\cap \calI_{n,j}\ne\emptyset} \sim \frac12\frac{d_n}{w_n^p}\frac{(\mathsf C_F\Gamma(\beta)\Gamma(1-\beta))^p}{\log d_n}. 
\eque
\end{Lem}
\begin{proof} Recall  in this case $\calH(n,K) = \calD_p^*(Kr_n)$ when $n$ is large enough and $r_n$ is now as in \eqref{eq:r_n critical}.
The first relation follows from Lemma \ref{lem:D_p*}. The second relation \eqref{eq:Cn2 critical} follows from   \eqref{eq:Cn2}.
For the third relation, we first
recall the the $p$-intersection of renewal processes, and $\wb F_p$ in \eqref{eq:wb F_p crit}. Thus, using a last-renewal decomposition, we have
\[
\proba\pp{R_{n,\vvi}\cap \calI_{n,1}\ne\emptyset} = \summ k1{d_n}\frac1{w_n^p}\wb F_p(d_n-k)\sim \frac{d_n}{w_n^p}\frac{(\mathsf C_F\Gamma(\beta)\Gamma(1-\beta))^p}{\log d_n}. 
\]
\end{proof}
The estimate of $\rho_n(r)$ in \eqref{eq:rho_n(r)}  is summarized below.
\begin{Lem}\label{lem:rho_n(r) critical}
Furthermore,
\[
\rho_n(r) \le \begin{cases}
\displaystyle \frac C{k_n^{\beta_r}\log d_n}+\frac C{\log ^2d_n}, & \mbox{ if } r = 1,\dots,p-1,\\
\\
\displaystyle \frac C{\log ^{3/2} d_n} + \frac {C \log k_n}{\log^2 d_n}, & \mbox{ if } r = p.
\end{cases}
\]
\end{Lem}
The proof is divided into the following three lemmas.
We first estimate the summation  ``off the diagonal'' in \eqref{eq:rho_n(r)}.
\begin{Lem}\label{lem:diagonal main}
We have for $r:=|\vvi\cap\vvi'| =1,\dots,p-1$,
\[
\sum_{\substack {j,j'=1,\dots,k_n\\|j-j'|\ge 1}} \proba \pp{R_{n,\vvi}\cap \calI_{n,j}\ne\emptyset, R_{n,\vvi'}\cap \calI_{n,j'}\ne\emptyset}\le 
 \frac C{\log^2 d_n}, 
\]
and   for $r=p$ (hence $\vv i= \vv i'$) 
\equh\label{eq:critical off diagonal r=p}
\sum_{\substack {j,j'=1,\dots,k_n\\|j-j'|\ge 2}} \proba \pp{R_{n,\vvi}\cap \calI_{n,j}\ne\emptyset, R_{n,\vvi}\cap \calI_{n,j'}\ne\emptyset}\le 
 \frac{C\log k_n}{\log^2d_n}.
\eque
\end{Lem}
\begin{proof}
Introduce 
\begin{equation}\label{eq:f l hat}
\what f_{n,\vvi,j} := \min(R_{n,\vvi}\cap\calI_{n,j}) \qmand \what \ell_{n,\vvi,j}:=\max(R_{n,\vvi}\cap\calI_{n,j}).
\end{equation}
Recall again that one can find a constant $C>0$ such that $u(n)\le Cn^{\beta-1}$ for all $n\in\N$. Write, by the first-renewal decomposition, 
\begin{align}
\proba & \pp{R_{n,\vvi}\cap \calI_{n,1}\ne\emptyset, R_{n,\vvi'}\cap \calI_{n,j+1}\ne\emptyset}=\sum_{k\in \calI_{n,1}}\sum_{k'\in\calI_{n,j+1}}\proba\pp{\what f_{n,\vvi,1} = k, \what  \ell_{n,\vvi',j}=k'}\nonumber\\
& = \sum_{k\in \calI_{n,1}}\proba\pp{\what f_{n,\vvi,1} = k}\sum_{k'\in\calI_{n,j+1}}u(k'-k)^{r}\frac1{w_n^{p-r}}\wb F_p((j+1)d_n-k')\nonumber\\
& \le C\sum_{k\in \calI_{n,1}}\proba\pp{\what f_{n,\vvi,1} = k}\sum_{k'\in\calI_{n,j+1}}(k'-d_n)^{(\beta-1)r}\frac1{w_n^{p-r}}\wb F_p((j+1)d_n-k')\nonumber\\
& = C \proba\pp{R_{n,\vvi}\cap \calI_{n,1}\ne\emptyset}\frac1{w_n^{p-r}}\sum_{k= (j-1)d_n+1}^{jd_n} k^{\beta_r-1} \wb F_p(jd_n-k).\label{eq:r=p}
\end{align}
If  $j\ge 2$, $r=1,\ldots,p$,
the summation above can be bounded as (recall $\beta_r<1$ always)  
\[
\sum_{k= (j-1)d_n+1}^{jd_n} k^{\beta_r-1} \wb F_p(jd_n-k) \le   ((j-1)d_n+1)^{\beta_r-1}   \sum_{k=0}^{ d_n-1} \wb F_p( k)\le C  \frac{ j^{\beta_r-1}d_n^{\beta_r}}{\log d_n}.
\]
If $j=1$, $r<p$, and $d_n$ is sufficiently large,   exploiting monotonicity we have
\begin{align}
\summ k1{d_n} k^{\beta_r-1} \wb F_p( d_n-k)
\notag
& \le \wb F_p(  d_n/2)\summ k1{\floor{d_n/2}} k^{\beta_r-1}  +  (d_n/2)^{\beta_r-1} \sum_{k=\floor{d_n/2}+1}^{  d_n} \wb F_p( d_n-k)\\
&\le C \frac{d_n^{\beta_r}}{\log d_n}\label{eq:convol bound}.
\end{align}
(Notice that we used $\beta_r<0$ above and hence $r=p$  is excluded.)
Therefore combining the above to \eqref{eq:r=p} with also \eqref{eq:rho_n crit}, we arrive at
\begin{align*}
\proba  \pp{R_{n,\vvi}\cap \calI_{n,1}\ne\emptyset, R_{n,\vvi'}\cap \calI_{n,j+1}\ne\emptyset}  & \le C\frac{\rho_n}{w_n^{p-r}}\frac{d_n^{\beta_r}j^{\beta_r-1}}{\log d_n} \\
& \le \frac{Cd_n^{\beta_r+1}}{n^{\beta_r+1}\log ^2d_n}j^{\beta_r-1}
\end{align*}
 for all (i) $r=1,\dots,p-1,j,n\in\N$ or (ii) $r=p, j\ge 2,n\in\N$,
for some constant $C>0$. Therefore, for $|\vvi\cap\vvi'| = r\le p-1$, 
\begin{align*}
\sum_{\substack {j,j'=1,\dots,k_n\\|j-j'|\ge 1}} \proba \pp{R_{n,\vvi}\cap \calI_{n,j}\ne\emptyset, R_{n,\vvi'}\cap \calI_{n,j'}\ne\emptyset} &\le C \frac{d_n^{\beta_r+1}}{n^{\beta_r+1}\log ^2d_n}\summ j1{k_n}(k_n-j)j^{\beta_r-1}\\
& \le C\frac{d_n^{\beta_r+1}k_n^{\beta_r+1}}{n^{\beta_r+1}\log^2 d_n} \le \frac C{\log^2d_n},
\end{align*}
and
\begin{align*}
\sum_{\substack {j,j'=1,\dots,k_n\\|j-j'|\ge 2}} \proba \pp{R_{n,\vvi}\cap \calI_{n,j}\ne\emptyset, R_{n,\vvi}\cap \calI_{n,j'}\ne\emptyset} &\le C \frac{d_n}{n\log ^2d_n}\summ j1{k_n}(k_n-j)j^{-1}\\
& \le C\frac{d_nk_n\log k_n}{n\log^2 d_n} \le \frac {C\log k_n}{\log^2d_n},
\end{align*}
as desired.
\end{proof}
With $r=p$ for adjoint blocks ($|j-j'|=1$), a little more care is needed, as the previous method leads to an inferior control with $r=p$. Below is an improved estimate in this case. 
\begin{Lem}\label{lem:critical diagonal}
We have
\[
\sum_{\substack{j,j'=1,\dots,k_n\\|j-j'| = 1}}\proba\pp{R_{n,\vvi}\cap\calI_{n,j}\ne\emptyset, R_{n,\vvi}\cap \calI_{n,j'}\ne\emptyset} \le
 \frac C{\log^{3/2}d_n}.
\]
\end{Lem}
\begin{proof}
 Introduce another sequence of integers $\{s_n\}_{n\in\N}$ such that $s_n\to\infty$ and $s_n/d_n\to 0$. 
For comparison purpose we write the proof for general $\vvi,\vvi'$ (so not necessarily $r=p$), and we shall see that at the end the method in this proof is superior than in the previous lemma {\em only} when $r=p$. 
Then recalling \eqref{eq:f l hat}, we proceed as
\begin{multline*}
\proba\pp{R_{n,\vvi}\cap \calI_{n,1}\ne\emptyset, R_{n,\vvi'}\cap \calI_{n,2}\ne\emptyset}\\
 \le  \sum_{k=1}^{d_n-s_n}\proba\pp{\what f_{n,\vvi,1} = k}\sum_{k'=d_n+1}^{2d_n}\proba\pp{\what\ell_{n,\vvi',2} = k'\mmid \what f_{n,\vvi,1} = k} + \proba\pp{\what f_{n,\vvi,1}\in\{d_n-s_n+1,\dots,d_n\}}.
\end{multline*}
We focus on the first double summation, which becomes
\begin{align*}
\sum_{k=1}^{d_n-s_n} & \proba\pp{\what f_{n,\vvi,1} = k}\sum_{k'=d_n+1}^{2d_n}u(k'-k)^{r(\beta-1)}\frac1{w_n^{p-r}}\wb F_p(2d_n-k') \\
& \le
\frac{C d_n}{n\log d_n}\sum_{k'=d_n+1}^{2d_n}s_n^{\beta_r-1}\frac1{w_n^{p-r}}\wb F_p(2d_n-k')   \le \frac{Cd_n}{n\log d_n}\frac{s_n^{\beta_r-1}}{w_n^{p-r}}\sum_{k=0}^{d_n-1} \wb F_p(k)  \le  \frac{Cd_n^2s_n^{\beta_r-1}}{nw_n^{p-r}\log^2d_n}, 
\end{align*}
where for the  previous line we have applied \eqref{eq:rho_n crit} and \eqref{eq:wb F_p crit}.
At the same time, 
\begin{align*}
\proba\pp{\what f_{n,\vvi,1}\in\{d_n-s_n+1,\dots,d_n\}} & \le  \sum_{k=d_n-s_n+1}^{d_n} \proba\pp{\what\ell_{n,\vvi,1}=k} \\
&= \sum_{k=d_n-s_n+1}^{d_n}\frac1{w_n^p}\wb F_p(d_n-k)\le \frac{Cs_n}{n\log s_n}.
\end{align*} 
Assume that $s_n$ grows at least at a polynomial rate. Then,  the optimal rate is achieved if  $d_n^2/(w_n^{p-r}\log d_n)\sim s_n^{2-\beta_r}$ (balancing the two bounds above), or equivalently
\[
s_n \sim \pp{\frac{d_n^2}{n^{\beta_r}\log d_n}}^{1/(2-\beta_r)}, r=1,\dots,p.
\]
(This sequence is indeed with polynomial growth.)
Therefore, with $s_n$ as above
\begin{align*}
\sum_{\substack{j,j'=1,\dots,k_n\\|j-j'| = 1}}\proba\pp{R_{n,\vvi}\cap\calI_{n,j}\ne\emptyset, R_{n,\vvi}\cap \calI_{n,j'}\ne\emptyset} 
& 
 \le \frac{Ck_n s_n}{n\log d_n} \sim \frac{Ck_n d_n^{2/(2-\beta_r)}}{n^{2/(2-\beta_r)}\log^{(3-\beta_r)/(2-\beta_r)}d_n}\\
& \sim C\pp{k_n^{\frac{\beta_r}{2-\beta_r}}\log^{\frac{3-\beta_r}{2-\beta_r}}d_n}\inv. 
\end{align*}
The desired result follows with $r=p$. Notice that with $r<p$, the above is inferior from what is obtained in Lemma \ref{lem:critical diagonal}.
\end{proof}
With Lemmas \ref{lem:diagonal main} and \ref{lem:critical diagonal} above, we have control over the summation  ``off the diagonal'' for $\rho_n$ in \eqref{eq:rho_n}. We are left with controlling the summation   ``on the diagonal''. 
\begin{Lem}\label{lem:critical same block}
For $r = |\vvi\cap\vvi'|=1,\dots,p-1$, 
\[
\proba\pp{R_{n,\vvi}\cap\calI_{n,1}\ne\emptyset, R_{n,\vvi'}\cap \calI_{n,1}\ne\emptyset} \le \frac C{k_n^{\beta_r+1}\log d_n}.
\]
\end{Lem}
\begin{proof}Recall that $\beta_r>\beta_p=0$.  It is clear that $\proba(R_{n,\vvi}\cap R_{n,\vvi'}\cap \calI_{n,1}\ne\emptyset)\le \sum_{k\in \calI_{n,1}} \proba(k\in R_{n,\vvi}\cap R_{n,\vvi'}) = d_n/w_n^{2p-r}$ which compared with the desired bound is of smaller order. So we focus on 
\[
\proba\pp{R_{n,\vvi}\cap \calI_{n,1}\ne\emptyset, R_{n,\vvi'}\cap \calI_{n,1}\ne\emptyset, R_{n,\vvi}\cap R_{n,\vvi'}\cap \calI_{n,1}=\emptyset},
\]which  can be bounded from above by
\begin{align*}
 & 2\sum_{1\le k<k'\le d_n} \proba(k\in R_{n,\vvi})\proba(\max R_{n,\vvi'} =k'\mid k\in R_{n,\vvi}) \\
 = & \frac{2 }{w_n^{p}}\sum_{1\le  k < d_n} (d_n-k)  \proba\pp{1\in R_{n,\vvi},\max R_{n,\vvi'} = 1+k}\\
\le & \frac C{w_n^{2p-r}}\sum_{k=1}^{d_n}(d_n-k)k^{\beta_r-1}\wb F_p(d_n-k) 
  \le C \frac{d_n^{\beta_r+1}}{w_n^{2p-r}\log d_n} \sim C \frac{1}{k_n^{\beta_r+1}\log d_n},
\end{align*}
where for the last inequality above we have applied the bound \eqref{eq:convol bound}.
\end{proof}
Now we have obtained Lemma \ref{lem:rho_n(r) critical} by combining Lemmas \ref{lem:diagonal main} to \ref{lem:critical same block}.

\begin{proof}[Proof of Lemma \ref{lem:2 crit}]
The proof is  similar to the proof of Lemma \ref{lem:2} by applying the two-moments method, and in particular, by verifying \eqref{eq:goal two moments}  based on the estimates in Lemmas \ref{lem:Cn critical} and  \ref{lem:rho_n(r) critical}   with the rate of $d_n$ specified in \eqref{eq:d_n critical}.   We omit the details, but only point out that the normalization constant $\mathfrak C_{F,p}$ is determined as (see the computation of $\limn\esp\wb\eta_{n,K}(E_s)$ in the proof of Lemma \ref{lem:2})
\[
\limn C_{n,1}(K)k_n\rho_n  =\frac12 K\frac{(\mathsf C_F\Gamma(\beta)\Gamma(1-\beta))^p}{p!(p-1)!} = K\mathfrak C_{F,p}. 
\]
\end{proof}
\subsection{Proof of Proposition \ref{prop:2}}

\begin{Lem}\label{lem:1 crit}
With  the setup for $\beta_p = 0$, the conclusion of Lemma \ref{lem:1} continues to hold.
\end{Lem}
\begin{proof} 
The proof is almost identical to the proof of Lemma \ref{lem:1}.  We only mention that  \eqref{eq:Omegan2} is replaced by
\begin{align*}
&\proba\pp{\exists  \vvi = (i_1,\dots,i_p)\in |\calD_p^*(Kr_n)|, \mbox{ s.t. } R_{n,\vvi}\ne\emptyset, i_1\le m } \\\le &  \frac{C}{\log(n)}       |\calD^*_{p-1}(Kr_n )|  
\le     \frac{C r_n\log^{p-2}(r_n)}{\log(n)}  \le  \frac{C}{ \log\log n }.
\end{align*}
\end{proof}
\begin{proof}[Proof of Proposition \ref{prop:2}]
The proof is similar to that of Proposition \ref{prop:1} in Section \ref{sec:approximation},  which follows from  Lemmas \ref{lem:1 crit} and \ref{lem:2 crit}.
\end{proof}

\subsection{Proof of \eqref{eq:thm crit}}

\begin{Lem}\label{lem:remainder1}
With the setup for $\beta_p = 0$,  the uniform control \eqref{eq:uniform_control'} continues to hold.
\end{Lem}
\begin{proof}%
The first part of the proof is as in the proof of Lemma \ref{lem:remainder}. With $\beta_p=0$, we are actually in the simpler case $X_{n,k}\topp K = \wt X_{n,k}\topp K$ therein, and we arrive at \eqref{eq:zeta} as before. That is, with
\[
 Z_{n,i_1,\dots,i_q}(K):= \max_{k=1,\dots,n}  r_n^{1/\alpha}\abs{\sum_{m<i_{q+1}<\cdots<i_p, [\vvi]>Kr_n} \frac{[\varepsilon_{\vvi_{q+1:p}}]}{  [\Gamma_{\vvi_{q+1:p}}]^{1/\alpha}} \inddd{k\in R_{n,\vvi}}},\ q=0,\dots,p-1,
\]
  it suffices to show that for each $q=0,\dots,p-1, i_1,\dots,i_q\le m$ fixed, 
\equh\label{eq:zeta0}
\lim_{K\to\infty}\limsupn \proba\pp{Z_{n,i_1,\dots,i_q}(K)>\epsilon} = 0 \mfa \epsilon>0.
\eque
Recall in the proof of Lemma \ref{lem:remainder} for the case $\beta_p<0$, the relation above was established by  an $L^2$ estimation directly. In the case $\beta_p=0$ here, we need a refinement (the previous estimate \eqref{eq:1} no longer works: now $nw_n^{-p}r_n\log^{p-q-1}(r_n) \sim C\frac{\log n}{(\log\log n)^q} \to\infty$). 
Instead, this time we first control
\[
\wt Z_{n,i_1,\dots,i_q}(K):= \max_{k=1,\dots,n}  r_n^{1/\alpha}\abs{\sum_{m<i_{q+1}<\cdots<i_p, [\vvi]>K\wt r_n} \frac{[\varepsilon_{\vvi_{q+1:p}}]}{  [\Gamma_{\vvi_{q+1:p}}]^{1/\alpha}} \inddd{k\in R_{n,\vvi}}}, q=0,\dots,p-1,
\]
with 
\[
\wt r_n :=r_n(\log n)^{\frac1{2/\alpha-1}} \gg r_n.
\]
The previous $L^2$-estimate \eqref{eq:1} now gives
\begin{align*}
\esp\wt Z_{n,i_1,\dots,i_q}(K)^2 & \le C nw_n^{-p}r_n^{2/\alpha} (K\wt r_n)^{1-2/\alpha}  \log^{p-q-1}(  \wt r_n)  \\
&\le C K^{1-2/\alpha}\frac{r_n}{\log n}\log^{p-q-1}(  r_n)\sim \frac{CK^{1-2/\alpha}}{(\log \log n)^{q}}.
\end{align*}
So \eqref{eq:zeta0} holds with $Z_{n,i_1,\dots,i_q}(K)$ replaced by $\wt Z_{n,i_1,\dots,i_q}(K)$. 

Next, we examine the difference
\equh\label{eq:diff}
\abs{\wt Z_{n,i_1,\dots,i_q}(K) - Z_{n,i_1,\dots,i_q}(K)}\le r_n^{1/\alpha} \max_{k=1,\dots,n}\abs{\sum_{\substack{m<i_{q+1}<\cdots<i_p\\
Kr_n\le [\vvi]\le K\wt r_n}}\frac{[\varepsilon_{i_{q+1:p}}]}{[\Gamma_{i_{q+1:p}}]^{1/\alpha}}\inddd{k\in R_{n,\vvi}}},
\eque
and consider the following event 
\[
\Omega_n\equiv\Omega_n(i_1,\dots,i_q):=\ccbb{\exists k=1,\dots,n:\sum_{m<i_{q+1}<\cdots<i_p\le K\wt r_n}\inddd{k\in R_{n,\vv i}}\ge 2}.
\]
The key observation is that in the event $\Omega_n^c$, for every $k=1,\dots,n$, the summation on the right-hand side of \eqref{eq:diff} has at most one non-zero term, and therefore, in the event $\Omega_n^c$, \eqref{eq:diff} becomes
\begin{align*}
\abs{\wt Z_{n,i_1,\dots,i_q}(K) - Z_{n,i_1,\dots,i_q}(K)}& \le r_n^{1/\alpha}\max_{\substack{m<i_{q+1}<\cdots<i_p\\Kr_n<[\vvi]\le K\wt r_n}} [\Gamma_{\vvi}]^{-1/\alpha}\\
& \le 
r_n^{1/\alpha} \max_{\substack{m<i_{q+1}<\cdots<i_p\\Kr_n<[\vvi]\le K\wt r_n}} (Kr_n)^{-1/\alpha} \frac{[\Gamma_{\vvi}]^{-1/\alpha}}{[\vvi]^{-1/\alpha}}\le K^{-1/\alpha}\sup_{i\in\N}\pp{\frac{\Gamma_i}{i}}^{-p/\alpha}.
\end{align*}
However, $\sup_{i\in\N}(\Gamma_i/i)^{-p/\alpha}$ is a finite random variable, and hence 
\[
\lim_{K\to\infty}\limsupn \proba\pp{\ccbb{\abs{\wt Z_{n,i_1,\dots,i_q}(K) - Z_{n,i_1,\dots,i_q}(K)}>\epsilon}\cap\Omega_n^c} = 0, \mfa \epsilon>0.
\]
It remains to examine $\Omega_n$. We have (recall that $i_1,\dots,i_q\le m$ are fixed),
\[
\proba(\Omega_n) \le \sum_{m< i_{q+1}<\cdots<i_{p+1}\le K\wt r_n}\proba\pp{R_{n,(i_1,\dots,i_{p+1})}\ne\emptyset} \le \binom{K\wt r_n}{p-q+1}\frac n{w_n^{p+1}} \le C\frac{\wt r_n^{p-q+1}}{w_n}\to 0.
\]
So combining   the two relations above completes the proof.
\end{proof}
\begin{proof}[Proof of \eqref{eq:thm crit}]
The proof is the same as for \eqref{eq:thm sub-crit} in Section \ref{sec:remainder}, with Proposition \ref{prop:1} and Lemma \ref{lem:remainder} replaced by Proposition \ref{prop:2} and Lemma \ref{lem:remainder1}, respectively.
\end{proof}

\begin{acks}[Acknowledgments]
The authors would like to thank Larry Goldstein, Takashi Owada and Gennady Samorodnitsky for very helpful discussions. The authors would like to thank an anonymous referee for the critical and yet constructive report that has helped us significantly during the revision. 
\end{acks}
\begin{funding}
The second author was partially supported by Army Research Office, USA (W911NF-20-1-0139). 
\end{funding}

\bibliographystyle{imsart-nameyear} 
\bibliography{references,references18}

\end{document}